\documentclass[reqno, 12 pt]{amsart}
\usepackage{amsfonts,amscd,mathrsfs}
\usepackage{amsthm}
\usepackage{amssymb}
\usepackage{latexsym}
\usepackage{url}
\usepackage{cite}
\usepackage{amsmath}
\usepackage{verbatim}
\usepackage{graphicx}
\usepackage{graphics}
\usepackage{color}
\usepackage{epsf}
\usepackage{enumerate}
\usepackage{geometry}
\usepackage{hyperref}
\definecolor{dark-blue}{rgb}{0,0,0.6}
\definecolor{Purple}{rgb}{0.2,0,0.25}
\hypersetup{breaklinks=true,
pagecolor=white,
colorlinks=true,
citecolor=dark-blue,
filecolor=black,%
linkcolor=Purple,%
urlcolor=black
}
\newtheorem{thm}{Theorem}[section]
\newtheorem{cor}[thm]{Corollary}
\newtheorem{lem}[thm]{Lemma}

\newtheorem{defin}[thm]{Definition}

\theoremstyle{definition}
\newtheorem{exmp}[thm]{Example}
\newtheorem{counterexample}[thm]{Counterexample}
\newtheorem{remark}[thm]{Remark}

\newcommand{\wt}{\widetilde}

\newcommand{\dom}{\textnormal{dom}}

\newcommand{\Int}{\textnormal{Int}}

\newcommand{\Span}{\textnormal{span}}

\newcommand{\R}{\mathbb{R}}

\newcommand{\C}{\mathbb{C}}
\newcommand{\N}{\mathbb{N}}

\newcommand{\M}{\mathscr{M}}

\newcommand{\scS}{\mathscr{S}}
\newcommand{\SUP}{\textnormal{SUP}}
\newcommand{\INF}{\textnormal{INF}}

\newcommand{\bref}[1]{\textbf{\ref{#1}}} 
\newcommand{\beqref}[1]{\textbf{(\ref{#1})}} 


\begin{document}
\date{March 29, 2020}

\title[Stability of the optimal values]{Stability of the optimal values under small perturbations of the constraint set}

\author{Daniel Reem}
\address{Daniel Reem, Department of Mathematics, The Technion -- Israel Institute of Technology, 3200003 Haifa, Israel.} 
\email{dream@technion.ac.il}
\author{Simeon Reich}
\address{Simeon Reich, Department of Mathematics, The Technion -- Israel Institute of Technology, 3200003 Haifa, Israel.} 
\email{sreich@technion.ac.il}
\author{Alvaro De Pierro}
\address{Alvaro De Pierro, CNP$\textnormal{q}$, Brazil}
\email{depierro.alvaro@gmail.com}

\maketitle

\begin{abstract}
This paper discusses a general and useful stability principle which, roughly speaking, says that given a uniformly continuous function defined on an arbitrary metric space, if the function is bounded on the constraint set and we slightly change this set, then its optimal (extreme) values on this set vary slightly, and, moreover, they are actually uniformly continuous as a function of the constraint set. The principle holds in a much more general setting than a metric space, since the distance function may be asymmetric, may attain negative and even infinite values, and so on. This stability  principle leads to applications in parametric optimization, mixed linear-nonlinear programming and analysis of Lipschitz continuity, as well as to a general scheme for tackling a wide class of non-convex and non-smooth optimization problems. We also discuss the issue of stability when the objective function is merely continuous. As a byproduct of our analysis we obtain a significant generalization of the concept of a generalized inverse of a linear operator and a very general variant of the so-called ``Hoffman's Lemma''. 
\end{abstract}

\section{Introduction}\label{sec:Introduction}
\subsection{Background:} The issue of stability in optimization problems has both theoretical and practical importance. More precisely, suppose that we are given an optimization problem consisting of a space $X$, an objective function (a target function) $f:X\to [-\infty,\infty]$, a constraint set $\emptyset \neq A\subseteq X$, and various parameters which influence the problem (for example, the parameters may influence the constraint set $A$, namely $A=A(t)$ for some fixed parameter $t$ in a parameter space; another example: parameters which define $f$, namely there is some function $g$ of two variables such that $f(x)=g(x,t)$, $x\in A$, where  $t$ is a parameter in a parameter space; in both cases one minimizes with respect to $x$). A natural question is what happens to the optimal values of the objective function $f$, as well as to its sets of minimizers and maximizers, when we slightly perturb some of the elements involved in the formulation of the problem. If the optimal values (and perhaps also the sets of minimizers and maximizers) change slightly as a result of the change in the elements of the problem, then the problem exhibits a certain kind of stability. Of course, the type of stability depends on the way in which we measure all the pertinent changes. 

Stability is a desired property which ensures that if one is able to control various types of imprecision which are inherent in many optimization problems (such as noise, inexact measurements, representation errors due to real parameters/expressions which are approximated by a finite decimal representation of them, and so on), then the resulting optimal values will not change a lot. 

There is a very long chain of research works that contain results related to the issue of stability in optimization problems, as can be seen by looking at the following (far from being exhaustive) list of references and the references therein: \cite{AnYen2018jour, BankGuddatKlatteKummerTammer1983book, Berge1963book, BonnansShapiro2000book, Borwein1986jour, DaniilidisGobernaLopezLucchetti2015jour, DantzigFolkmanShapiro1967jour, DempeMehlitz2015jour, DinhGobernaLopez2012jour, Dolecki1978jour, Fiacco1983book, FiaccoIshizuka1990jour, KlatteKummer1985incol, LeePham2016jour, LeeTamYen2012jour, LucVolle2018jour, LudererMinchenkoSatsura2002book,  MordukhovichNamYen2009jour, MoussaouiSeeger1994jour, PallaschkeRolewicz1997book, Robinson1987jour, Rockafellar1970book, Rockafellar1982jour, ShanHanHuang2015jour, TamNghi2018jour, Thibault1991jour, WalkupWets1969jour, Wets2003jour, Yen1997jour, ZlobecGardnerBen-Israel1982incol}. These works consider a variety of optimization problems, in various settings (often Euclidean spaces, but sometimes also in infinite-dimensional Banach spaces or other spaces). They establish certain properties related to either the optimal values of the objective function or its sets of minimizers and maximizers, under suitable assumptions on the objective function and the structure of the set of constraints, such as convexity, linear or piecewise linear structure, quadratic structure, compactness, and so on. Frequently the established properties are semicontinuity (lower or upper) and closedness; occasionally (usually in finite-dimensional Euclidean spaces)  stronger properties are established, such as continuity, Lipschitz continuity, convexity,  and differentiability. The aforementioned  results are interpreted as being stability results.

In the above-mentioned cases the perturbations occur in the space of parameters which define either the set of constraints or the objective function. This is one of the reasons why the corresponding research field to which these works belong is traditionally called ``parametric optimization'', although they can also be regarded as belonging to ``variational analysis'' \cite{Mordukhovich2006Ibook, RockafellarWets1998book}. The parameters themselves are frequently real numbers or vectors in a linear space endowed with a norm or a topology, but sometimes they belong to a different entity. For instance, in \cite{Artstein1994jour, Romisch2003incol} the  parameter is a probability measure; in \cite{RobinsonWets1987jour} the parameter is a pair consisting of a point in a topological space and a probability measure; in  \cite{Vogel1994jour, Vogel2005jour} the parameter is a sample belonging to a probability space and one discusses various probabilistic types of convergence, as well as convergence in the sense of inner or outer limits (limits which are related to the Kuratowski-Painlev\'e notion of convergence). For related but somewhat different  types of stability, see, for example, \cite{ArtsteinWets1994jour, AttouchWets1993jour-a, DinhGobernaLopez2010jour, Pinelis2019jour, PhuYen2001jour, ReemReich2018jour, Tihonov1966jour, Zaslavski2010book, Zaslavski2013book}. 

\subsection{Our stability principle:} 
In this paper we discuss a different type of stability, which, in our opinion, is not less natural than other types of stability discussed in the literature. More precisely, after some preliminaries (Section \bref{sec:Preliminaries}) we present in Section \bref{sec:StabilityPrinciple} our stability  principle which, roughly speaking, says the following: given a uniformly continuous function defined on an arbitrary metric space, if the function is bounded on the constraint set and we slightly change this set, where the change is measured with respect to the Hausdorff distance, then its optimal (extreme) values on this set vary slightly.

One can think of the above-mentioned stability principle as a continuous dependence result, where  the resulting continuous functions are the infimum and the supremum of the objective function over the constraint set. This is how we actually formulate the stability principle (Theorem \bref{thm:HausdorffSupInf}). The advantage of  this point of view is that it enables one to establish other properties of the infimum and supremum functions, such as their uniform continuity and, sometimes, their Lipschitz continuity. Such properties may not have a very intuitive interpretation. Theorem \bref{thm:HausdorffSupInf}, which is new to the best of our knowledge (but see Remark \bref{rem:FurtherCommentsStability}\beqref{rem:FurtherComparisonLiterature} below for some predecessors and other related results), has a  potential to extend, and perhaps to re-prove, some of the stability results mentioned earlier (for an illustration, see Section \bref{sec:LinearNonlinear}). It seems to be especially promising for finite-dimensional normed spaces, manifolds, and the like, since, as is well known, a continuous function defined on a closed and bounded (hence compact) subset of such spaces is uniformly continuous on this subset. Another relevant example is continuous linear functionals on arbitrary normed spaces (since, as is well known, they are automatically Lipschitz continuous). 

Interestingly, our stability result holds in a much more general setting than metric spaces, since we assume essentially nothing on the distance function: in particular, it can be negative, infinite, asymmetric, and so on; this simple observation is, in fact, a key contribution. As a result, our ``topology-free'' stability principle holds a promise to be relevant to many scenarios in the literature in which the distance function is not a metric, as we illustrate below and in Section \bref{sec:LinearNonlinear}. 

First example: the distance function is the well-known Bregman divergence/distance \cite{Bregman1967jour} (see also \cite{ReemReichDePierro2019jour(BregmanEntropy)} for a recent semi-survey and an extensive re-examination) or one of its generalizations \cite{Reem2012incol};  second example: distortion measures, divergences, and distance functions used in information theory, data analysis, data processing, machine learning and the like \cite{AnthonyRatsaby2018jour, Basseville2013jour,  Gray2013book, Kogan2007book, Koknar-TezelLatecki2009inproc}; third example: Finsler quasi-metric spaces \cite{Tamassy2008jour}; fourth example:  distance functions which appear in fixed point theory  \cite{KadelburgRadenovic2014jour, Khamsi2015jour, KirkShahzad2014book, ReemReich2009jour}; fifth example: distances induced by quasi-norms \cite{Kalton2003incol}; sixth example: distances induced by asymmetric norms (such as asymmetric Minkowski functionals) \cite{Cobzas2013book, Minkowski1911}; seventh example: distances induced by star bodies  \cite{Cassels1997, GruberLek, Mahler1946, Mordell1945}; eighth example: numerous distance functions  which appear in many scientific and technological areas \cite{DezaDeza2016book}; ninth example: the setting is the energy levels of an atom, namely the space is $X:=\N$, and $d(x,y)$ denotes the amount of energy that one needs to invest  in order to bring an electron from energy level $x\in X$ to energy level $y\in X$ (in this case $d(x,y)<0$ means that one gains energy when the electron jumps from energy level $x$ to energy level $y$).

The assumptions made in the formulation of Theorem \bref{thm:HausdorffSupInf} are essential, as we illustrate in Section \bref{sec:StabilityPrinciple}  by using some counterexamples. Despite this, we are able to formulate a counterpart to Theorem \bref{thm:HausdorffSupInf}, namely Theorem \bref{thm:HausdorffSupInfInfinite}, in which the uniform continuity can be replaced by mere continuity, provided the optimal values are infinite.  

\subsection{Additional major contributions: } 
Sections \bref{sec:ParamOpt}--\bref{sec:NonconvexNonsmooth} of the paper present additional major contributions, among them several applications of Theorem \bref{thm:HausdorffSupInf}, as well as new and seemingly unrelated concepts and results. One application (Section \bref{sec:ParamOpt}) is in parametric optimization, where we prove a general continuity result related to the so-called ``optimal value function''. A second application appears in Section \bref{sec:LinearNonlinear}, where we apply the result of Section \bref{sec:ParamOpt} to mixed linear-nonlinear programming problems. As a byproduct of our analysis we obtain a significant generalization of the concept of a generalized inverse of a linear operator and a very general variant of the so-called ``Hoffman's Lemma''. A third application is presented in Section \bref{sec:ContinuityLipCont}, where we obtain a sequence of Lipschitz constants related to certain functions. This latter application has recently been used in \cite{ReemReichDePierro2019teprog} in the context of estimating the rate of convergence of a certain first order proximal gradient method.  A fourth application  (Section \bref{sec:NonconvexNonsmooth}) is a general scheme for tackling a wide class of non-convex and non-smooth optimization problems.

\section{Preliminaries}\label{sec:Preliminaries}
Our setting is an arbitrary nonempty set $X$ endowed with an arbitrary function $d:X^2\to[-\infty,\infty]$. We refer to the set $X$ as the ``space'', to $d$ as the ``distance function'' or the ``pseudo-distance'', and to $(X,d)$ as the ``pseudo-distance space''. Note that $d$ is, in general, not a metric, for example because it can attain negative and even infinite values, it may be asymmetric, and so on, and hence we use the notion ``pseudo-distance space''. We denote by $2^X$ the set of all subsets of $X$.  Given $\psi:X\to[-\infty,\infty]$, its effective domain is the set $\dom(\psi):=\{x\in X: \psi(x)\in\R\}$. Given an arbitrary function $f:X\to \R$, we denote by $\SUP_f: 2^X\to [-\infty,\infty]$ and $\INF_f:2^X\to [-\infty,\infty]$ the supremum and infimum operators, respectively, associated with $f$. These operators are naturally defined by $\SUP_f(A):=\sup\{f(a): a\in A\}$ and $\INF_f(A):=\inf\{f(a):  a\in A\}$ for each $A\in 2^X$, respectively, where we use the usual convention that $\sup \emptyset:=-\infty$ and $\inf \emptyset:=\infty$. We have $\dom(\SUP_f)\neq\emptyset$ and $\dom(\INF_f)\neq\emptyset$ because they contain all the singletons $\{x\}$, $x\in X$, by the assumption that $f(x)\in\R$ for each $x\in X$.

Given a point $x\in X$ and a nonempty subset $A$ of $X$, the distance between them is defined by $d(x,A):=\inf\{d(x,a): a\in A\}$. Given $r>0$, we denote $B[A,r]:=\{x\in X: d(x,A)\leq r\}$ and refer to this set as the ``closed ball of radius $r$ around $A$'' (this ``ball'' always contains $A$ if $d$ is a metric, but in general it may be empty). Given another $\emptyset\neq A'\subseteq X$, we denote the standard distance between $A$ and $A'$ by 
\begin{equation}\label{eq:d(A,A')}
d(A,A'):=\inf\{d(a,a'): a\in A, a'\in A'\}. 
\end{equation}
We define the asymmetric Hausdorff distance between $A$ and $A'$ by 
\begin{equation}\label{eq:D_asyH}
D_{asyH}(A,A'):=\sup\{d(a,A'): a\in A\},
\end{equation}
and the (symmetric) Hausdorff distance (also called the ``Pompeiu-Hausdorff distance'') between $A$ and $A'$ by  
\begin{equation}\label{eq:D_H}
D_H(A,A'):=\max\{D_{asyH}(A,A'), D_{asyH}(A',A)\}. 
\end{equation}
It is a simple consequence of \beqref{eq:d(A,A')}, \beqref{eq:D_asyH} and \beqref{eq:D_H} that $d(A,A')$, $D_{asyH}(A,A')$ and $D_H(A,A')$ belong to $[-\infty,\infty]$ and $d(A,A')\leq D_{asyH}(A,A')\leq D_H(A,A')$. 

Roughly speaking, the Hausdorff distance quantifies the ``degree of  similarity'' between $A$ and $A'$. More precisely, if (as it happens, for example, in the real world where $d$ is the usual Euclidean distance) there is a resolution parameter $r>0$, namely a positive number $r$ with the property that we cannot distinguish, by using distance measurements, between two points $x$ and $y$ in $X$ if $d(x,y)<r$, and if $D_H(A,A')<r$, then we cannot distinguish between $A$ and $A'$ by using distance measurements. Indeed, the inequality $D_H(A,A')<r$ implies that for each $a\in A$, there is $a'\in A'$ such that $d(a,a')<r$ (namely, no point in $A$ can be distinguished from a point in $A'$) and for each $a'\in A'$,  there is $a\in A$ such that $d(a',a)<r$ (namely, no point in $A'$ can be distinguished from a point in $A$). We note that in general (if, say, $d(x,y)=\infty$ for all $(x,y)\in X^2$), given $\delta>0$ and $\emptyset\neq A\subseteq X$, there may not be any $\emptyset\neq A'\subseteq X$ such that $D_H(A,A')<\delta$; if, however, $d$ is a metric, then we always have $0=D_H(A,A)<\delta$.  

In the sequel we consider various continuity notions of real functions. These notions are word for word as in the usual (metric) case, although in the general case their interpretation is less intuitive. More precisely, let $f:X\to\R$ and $\emptyset\neq A\subseteq X$ be given. We say that $f$ is continuous at $x\in X$ if for every $\epsilon>0$, there exists $\delta>0$ such that for each $y\in X$ satisfying $d(x,y)<\delta$, we have $|f(x)-f(y)|<\epsilon$. We say that $f$ is continuous on $A$ if it is continuous at each $x\in A$. We say that $f$ is uniformly continuous on $A$ if for each $\epsilon>0$, there exists $\delta>0$ such that for all $(x,y)\in A^2$ satisfying $d(x,y)<\delta$, we have $|f(x)-f(y)|<\epsilon$.  Given another pseudo-distance space $(I,d_I)$ and $\psi:X\to I$, we say that $\psi$ is Lipschitz continuous on $A\subseteq X$ with a constant $0\neq\Lambda\in \R$ (or, briefly, that $\psi$ is $\Lambda$-Lipschitz continuous on $A$) if $d_I(\psi(x),\psi(y)\leq\Lambda d(x,y)$ for all $(x,y)\in A^2$ (here and elsewhere, if $d(x,y)\in\R$ for all $(x,y)\in A^2$, then the case $\Lambda=0$ is legitimate too). 

Given a family $(A_t)_{t\in I}$ of nonempty subsets of $X$ and a point $t_0\in I$, we say that $\limsup_{t\to t_0}D_{H,X}(A_{t_0},A_t)\leq 0$ if for every $\epsilon>0$, there exists $\delta>0$ such for all $t\in I$ which satisfy $d_I(t_0,t)<\delta$, we have $D_{H,X}(A_{t_0},A_t)<\epsilon$; here, for the sake of clarity, we denoted $D_H$ by $D_{H,X}$, to emphasize that  it is induced by $d$ and not by $d_I$. 

We finish this section with a few remarks. First, the fact that the pseudo-distance $d$ may be asymmetric, means that the above-mentioned types of continuity, as well as other notions (such as the distance $d(x,A)$ between some $x\in X$ and $\emptyset\neq A\subseteq X$, as well as the ``ball'' $B[A,r]$, $r>0$) will be different if, in the corresponding definitions, we change the order in which we measure the distance between two given points. For example, in the definition of continuity, instead of assuming the inequality $d(x,y)<\delta$ we could have assumed that $d(y,x)<\delta$. Nevertheless, in either case the relevant proofs follow  essentially the same reasoning. Second, we could extend our setting and some of the results even further (for example, by assuming that $d:X^2\to L$, where $L\neq\emptyset$ is an arbitrary linearly/totally/simply ordered set), but we refrain from doing this here.

\section{Stability of the optimal values}\label{sec:StabilityPrinciple}
In this section we present the stability principle in both its finite and infinite versions. Despite the general setting, the proofs are rather simple. We present them for the sake of completeness and in order to  eliminate any suspicion that possibly a few subtle points have been missed due to the general setting. The counterexamples which come afterward show that the assumptions imposed in the formulation of the principle are essential. At the end of the section we make a few relevant comments regarding this stability principle.

\begin{thm}\label{thm:HausdorffSupInf}{\bf (Stability of the optimal values: the finite case)}
Let $(X,d)$ be a pseudo-distance space and let $f:X\to\R$ be given. If $f$ is uniformly continuous on $(X,d)$ in the sense of Section \bref{sec:Preliminaries}, then $\SUP_f$ is uniformly continuous on the pseudo-distance space $(\dom(\SUP_f),D_H)$, and $\INF_f$ is uniformly continuous on the pseudo-distance space $(\dom(\INF_f),D_H)$. If $f$ is $\Lambda$-Lipschitz continuous for some $\Lambda>0$ (in the sense of Section \bref{sec:Preliminaries}), then $\SUP_f$ is $\Lambda$-Lipschitz continuous on $(\dom(\SUP_f),D_H)$ and $\INF_f$ is $\Lambda$-Lipschitz continuous on $(\dom(\INF_f),D_H)$. 
\end{thm}

\begin{proof}
We only consider the assertions regarding $\SUP_f$, because the assertions regarding $\INF_f$ can be proved in a similar manner, or can be deduced from the  assertions regarding $\SUP_f$ by working with $-f$ instead of $f$. 

Fix some arbitrary $\epsilon>0$. Assume first that $f$ is uniformly continuous. Then there is $\delta>0$ such that for all $(x,y)\in X^2$ which satisfy $d(x,y)<\delta$, we have $|f(x)-f(y)|<0.5\epsilon$. We first prove that $\SUP_f$ is uniformly continuous on its effective domain, where we actually show that the $\delta$ from the previous lines can be associated, in the definition of the uniform continuity of $\SUP_f$, with the given $\epsilon$. Let $A, A'\in\dom(\SUP_f)$ be arbitrary such that $D_H(A,A')<\delta$ (if no such sets exist, then the proof is complete, vacuously). Since $\SUP_f(A)\in\R$, there is some $a\in A$ such that 
\begin{equation}\label{eq:s_A-epsilon<f(a)}
\SUP_f(A)-0.5\epsilon<f(a).
\end{equation}
From \beqref{eq:D_H} we have $d(a,A')\leq D_H(A,A')<\delta$. Hence there is  some $a'\in A'$ such that $d(a,a')<\delta$. The uniform continuity of $f$ and \beqref{eq:s_A-epsilon<f(a)} imply that $\SUP_f(A)-0.5\epsilon<f(a)<f(a')+0.5\epsilon$. Since $-\infty<f(a')\leq \SUP_f(A')$, we obtain $\SUP_f(A)-\SUP_f(A')<\epsilon$. By interchanging the roles of $A$ and $A'$ we obtain, using an argument similar to the above one, that $-\epsilon<\SUP_f(A)-\SUP_f(A')$. 

Now assume that $f$ is $\Lambda$-Lipschitz continuous for some $\Lambda>0$. Fix some arbitrary  $A, A'\in \dom(\SUP_f)$. We need to show that $|\SUP_f(A)-\SUP_f(A')|\leq \Lambda D_H(A,A')$. The assertion is obvious if $D_H(A,A')=\infty$, and so from now on we may assume that $D_H(A,A')<\infty$. We observe that the case $D_H(A,A')=-\infty$ is impossible since the fact   that $f$ is $\Lambda$-Lipschitz continuous implies that $0\leq |f(x)-f(y)|\leq \Lambda d(x,y)$ for all $(x,y)\in X^2$, and so the assumption that and $\Lambda>0$ implies that $d$ and hence $D_H$ are nonnegative. Thus we can actually assume that $D_H(A,A')\in [0,\infty)$. The definition of $\SUP_f(A)$  implies that \beqref{eq:s_A-epsilon<f(a)} holds for some $a\in A$. Since from \beqref{eq:D_H} we have $d(a,A')\leq D_H(A,A')<D_H(A,A')+0.5\epsilon$, there is some $a'\in A'$ such that $d(a,a')<D_H(A,A')+0.5\epsilon$. We conclude from \beqref{eq:s_A-epsilon<f(a)} and  the $\Lambda$-Lipschitz  continuity of $f$ that $\SUP_f(A)-0.5\epsilon<f(a)\leq f(a')+\Lambda d(a,a')\leq f(a')+\Lambda D_H(A,A')+0.5\Lambda\epsilon$. As obviously $f(a')\leq \SUP_f(A')$, it follows that $\SUP_f(A)-\SUP_f(A')<\Lambda D_H(A,A')+0.5\epsilon(1+\Lambda)$. Since $\epsilon$ can be arbitrarily small, we have $\SUP_f(A)-\SUP_f(A')\leq\Lambda D_H(A,A')$. By interchanging the roles of $A$ and $A'$ we obtain, using an argument similar to the above one, that $-\Lambda D_H(A,A')\leq\SUP_f(A)-\SUP_f(A')$. 
\end{proof}

\begin{thm}\label{thm:HausdorffSupInfInfinite}{\bf (Stability of the optimal values: the infinite case)}
Let $(X,d)$ be a pseudo-distance space. Suppose that $f:X\to\R$ is continuous in the sense of Section \bref{sec:Preliminaries}. Given an arbitrary $\emptyset \neq A\subseteq X$, if $\SUP_f(A)=\infty$, then for all $\mu\in\R$, there exists $\delta>0$ such that for each nonempty subset $A'\subseteq X$ satisfying  $D_{asyH}(A,A')<\delta$ (in particular, for each  $\emptyset\neq A'\subseteq X$ satisfying  $D_H(A,A')<\delta$), the following inequality holds:
\begin{equation}\label{eq:supM}
\mu<\SUP_f(A').
\end{equation}
 Similarly, if $\INF_f(A)=-\infty$, then for all $\mu\in\R$, there exists $\delta>0$ such that for each  nonempty subset $A'\subseteq X$ satisfying  $D_{asyH}(A,A')<\delta$ (in particular, for each  $\emptyset\neq A'\subseteq X$ satisfying  $D_H(A,A')<\delta$), one has
\begin{equation}\label{eq:infM}
\INF_f(A')<\mu.
\end{equation}
\end{thm}

\begin{proof}
We only consider the case where $\SUP_f(A)=\infty$, because the proof in the case where $\INF_f(A)=-\infty$  employs similar arguments (or, alternatively, can be deduced from the first case by taking $-f$ instead of $f$). Let $\mu\in\R$ and $\epsilon>0$ be arbitrary. The definition of $\SUP_f(A)$ and the fact that $\SUP_f(A)=\infty$ imply that there exists a point $a\in A$ such that \begin{equation}\label{eq:M<f(a)}
\mu+\epsilon<f(a).
\end{equation}
Since $f$ is continuous on $X$, it is continuous at $a$. Thus for the given $\epsilon$, there exists $\delta>0$ such that for every $x\in X$ satisfying $d(a,x)<\delta$, we have 
\begin{equation}\label{eq:f(a)<=f(x)+epsilon}
f(a)<f(x)+\epsilon. 
\end{equation}
Let $\emptyset\neq A'\subseteq X$ be arbitrary such that $D_{asyH}(A,A')<\delta$. If no such subset $A'$ exists, then the proof is complete (the assertion holds vacuously). Otherwise, the inequality $D_{asyH}(A,A')<\delta$ and \beqref{eq:D_asyH} imply that $d(a,A')\leq D_{asyH}(A,A')<\delta$. Therefore there is a point $a'\in A'$ such that $d(a,a')<\delta$. We conclude from \beqref{eq:M<f(a)} and  \beqref{eq:f(a)<=f(x)+epsilon} (with $x:=a'$) that $\mu+\epsilon<f(a)<f(a')+\epsilon$, namely $\mu<f(a')$. Since obviously $f(a')\leq \SUP_f(A')$, we have $\mu<\SUP_f(A')$, that is, \beqref{eq:supM} holds. 
\end{proof}

\begin{counterexample}\label{counterex:CounterexampleHausdorff(R)}
The uniform continuity assumption on $f$ in Theorem \bref{thm:HausdorffSupInf} is essential. Indeed, let $X:=\R$ and let $d$ be the standard absolute value metric; let $A:=\cup_{k=2}^{\infty}[2k,2k+1]$. For each $1<j\in\N$ and $1<k\in \N$, let $J_{j,k}$ be defined by $J_{j,k}:=[2k,2k+1+(1/k)]$ if $j=k$ and $J_{j,k}:=[2k,2k+1]$ if $j\neq k$. For every  $1<j\in\N$, let $A_j:=\cup_{k=2}^{\infty}J_{j,k}$. Let $\N_1:=\N\backslash\{1\}$, and let $f:X\to\R$ be defined as follows for each $t\in X$: 
\begin{equation*}
\!f(t):=\!\!\left\{
\begin{array}{lll}
\!\!\! 0, & \!\!\!\textnormal{if}\,\, t\in (-\infty,4], \\
\!\!\! 0, & \!\!\!\textnormal{if}\,\, t\in [2k,2k+1]\,\,\textnormal{for some}\,\,k\in\N_1,\\
\!\!\!-2kt+2k(2k+1), &\!\!\! \textnormal{if}\,\, t\in [2k+1,2k+1+\frac{1}{2k}]\,\,\textnormal{for some}\,k\in\N_1,\\
\!\!\! 4kt-4k(2k+1)-3, &\!\!\! \textnormal{if}\,\, t\in [2k+1+\frac{1}{2k},2k+1+\frac{1}{k}]\,\textnormal{for some}\,k\in\N_1,\\
\!\!\!\displaystyle{\frac{kt}{1-k}}-\displaystyle{\frac{k(2k+2)}{1-k}}, &\!\!\! \textnormal{if}\, t\in [2k+1+\frac{1}{k},2k+2]\,\,\textnormal{for some}\,k\in\N_1.
\end{array}\right.
\end{equation*}
Then $f(t)$ is defined for all $t\in X$ and $f$ is continuous, but not uniformly continuous on $X$. In addition, for each $1<j\in\N$, we have $\INF_f(A_j)=-1$ and $\SUP_f(A_j)=1$. Since  $\INF_f(A)=0=\SUP_f(A)$ and $\lim_{j\to\infty}D_H(A,A_j)=0$, we conclude that neither $\INF_f$ nor $\SUP_f$ are continuous at $A$. 
\end{counterexample}

\begin{counterexample}\label{counterex:CounterexampleHausdorff(l2)}
Here is another counterexample related to Theorem \bref{thm:HausdorffSupInf}, where now $X$ is bounded. For each $k\in\N$, let $e_k\in\ell_2$ be  the vector having 1 in its $k$-th component and 0 in the other components, and let $[0,e_k]$ be the line segment which connects the origin 0 of $\ell_2$ with $e_k$. For each $j,k\in\N$, define $J_{j,k}:=[0,((k+1)/k)e_k)$ (half-open line segment) if $j=k$ and $J_{j,k}:=[0,e_k]$ if $j\neq k$. Let $X$ be defined by $X:=\cup_{k=1}^{\infty}[0,((k+1)/k)e_k)$ and $d$ be the metric induced by the $\ell_2$ norm. Let $A:=\cup_{k=1}^{\infty}[0,e_k]$ and for all $j\in \N$, let $A_j:=\cup_{k=1}^{\infty}J_{j,k}$. Then $X$, $A$ and $A_j$, $j\in\N$, are bounded, and we have $A\subset X$ and also $A_j\subset X$ for each $j\in\N$. For each $k\in\N$, let $f_k:[0,(k+1)/k)\to\R$ be defined by  $f_k(t):=0$ for every $t\in [0,1]$ and $f_k(t):=\sin(2\pi/(1+(1/k)-t))$ for all $t\in [1,(k+1)/k)$. Now let $f:X\to\R$ be defined as follows: given $0\neq x=(x_i)_{i=1}^{\infty}\in X$, there exists a unique $k\in\N$ such that $x\in [0,((k+1)/k)e_k)$, namely $x_i=0$ if $i\neq k$ and $x_i\in [0,(k+1)/k)$ if $i=k$, where $i,k\in\N$; in this case, define $f(x):=f_k(x_k)$; for $x=0$, let $f(x):=0$. Then $f$ is continuous on $X$, but it is not uniformly continuous there. Moreover, $\INF_f(A)=0=\SUP_f(A)$. However, for each $j\in\N$, we have $\INF_f(A_j)=-1$ and $\SUP_f(A_j)=1$, and, in addition, $\lim_{j\to\infty}D_H(A,A_j)=0$. Thus neither $\INF_f$ nor $\SUP_f$ are continuous at $A$.
\end{counterexample}

\begin{remark}\label{rem:FurtherCommentsStability}
\begin{enumerate}[(a)]
\item {\bf The level of generality of $d$:} Both Theorem \bref{thm:HausdorffSupInf} and Theorem \bref{thm:HausdorffSupInfInfinite} hold even though we assume essentially nothing regarding $d$. In particular, $d$ may attain infinite values, may be negative, may not satisfy the triangle inequality and so on.
\item\label{item:redundant} {\bf A redundant assumption: } In the proof of Theorem \bref{thm:HausdorffSupInf} we required both $A$ and $A'$ to be in $\dom(\SUP_f)$. This requirement is actually redundant, since if $A\in \dom(\SUP_f)$ and $\emptyset \neq A'\subseteq X$ satisfies $D_H(A,A')<\delta$ (where $\delta$ is associated with some given $\epsilon>0$ in the definition of uniform continuity of $f$), then for a fixed $a'\in A'$ we can find $a\in A$ such that $d(a',a)<\delta$, and so the uniform continuity of $f$ implies that $-\infty<f(a')<f(a)+\epsilon\leq \SUP_f(A)+\epsilon$; in other words, $A'\in \dom(\SUP_f)$. 
\item {\bf An intuitive interpretation: } The previous part implies an intuitive interpretation of Theorem \bref{thm:HausdorffSupInf}. Indeed, suppose that we are given a real function $f$ which is defined on  an arbitrary pseudo-distance space $(X,d)$, and suppose that we know that $f$ is  uniformly  continuous on $X$; given a nonempty subset $A$ of $X$ on which $f$ is bounded from above, suppose that we perturb $A$ slightly, where the perturbation  is measured with respect to the Hausdorff distance; let $A'$ be the new subset which is obtained from the original subset $A$; then $f$ is bounded from above on $A'$ and its supremum over $A'$ is almost equal to the supremum of $f$ over $A$. Similarly, if $f$ is bounded from below, then the infimum of $f$ over $A'$ is almost equal to the infimum of $f$ over $A$. Moreover, if $f$ is known to be Lipschitz  continuous, then the perturbation of the optimal values ``behaves better''. 
\item\label{sec:StabilityMinMax}{\bf Stability of the sets of minimizers and maximizers: }
Theorem \bref{thm:HausdorffSupInf} raises the corresponding question regarding stability of the sets of minimizers and maximizers of the  function under consideration, that is, not only the stability of its optimal values. In general, stability does not hold as is shown in the counterexample in the next paragraph. However, under further assumptions a weak stability principle related to the minimizers and maximizers can be formulated: Roughly speaking, given a continuous function $f:X\to\R$, where $(X,d)$ is a compact metric space, and given some nonempty and closed subset $A$ of $X$, if one slightly perturbs $A$ to a new nonempty and closed subset $A'$, then the set of minimizers of $f$ over $A$ is ``slightly'' perturbed too in the sense that the asymmetric Hausdorff distance (and hence the standard distance) between the set of minimizers of $f$ over $A'$ and the set of minimizers of $f$ over $A$ is small; an analogous assertion holds  regarding the perturbed and original  sets of maximizers. This result is essentially known \cite[Theorem 4.5]{Kummer1977jour}, and its proof is rather simple. 

As for the promised counterexample, consider the case where $X:=[-20,20]$ with the absolute value metric. Let $A:=[0,\pi]$ and $A'_{\epsilon}:=[-\epsilon,\pi-\epsilon]$ for every $\epsilon\in (0,1/2)$. In addition, let $f:X\to\R$ be defined by $f(t):=|\sin(t)|$, $t\in X$. Then $f$ is uniformly continuous on $X$ and its set of minimizers over $A$ is $\{0,\pi\}$. However, no matter how small $\epsilon$ is, the set of minimizers of $f$ over $A'_{\epsilon}$ is $\{0\}$. Thus the Hausdorff distance between these sets is $\pi$, despite the fact that $D_H(A'_{\epsilon},A)=\epsilon\to 0$ as $\epsilon$ tends to $0$.  On the other hand, the stability result stated in the previous paragraph does hold because the asymmetric Hausdorff distance between $\{0\}$ and $\{0,\pi\}$ is 0.  
\item\label{rem:FurtherComparisonLiterature}
{\bf Predecessors of our stability principle: }
Here we discuss a few predecessors of Theorem \bref{thm:HausdorffSupInf},  and also additional related results. We mention them in the next paragraphs, but before doing so we note that there are several significant differences between these results and our ones, for instance with respect to the setting, with respect to the proofs, the fact that we present explicit estimates and more; we also note that we have not been aware of these results when we derived Theorem \bref{thm:HausdorffSupInf}. 

The first relevant result is a sequential continuity (not uniform continuity or Lipschitz continuity as in our Theorem \bref{thm:HausdorffSupInf}) result regarding the maximal value function in the setting of metric spaces, which seems to be implicit in Kummer \cite{Kummer1977jour}. Interestingly, it seems that this fact has not been observed so far in the literature, not even in \cite{Kummer1977jour} (the explicit relevant results there are for compact metric spaces \cite[Theorem 4.5]{Kummer1977jour}, finite-dimensional Euclidean spaces with a quasiconvex objective function \cite[Theorem 4.10]{Kummer1977jour}, and Banach spaces \cite[Theorem 4.11]{Kummer1977jour} with quite demanding and technical assumptions). 

Another relevant result appears in \cite[Satz 2.1 and Section 4]{Kummer1977a-jour}. It is a sequential continuity result related to the maximal value function of a quadratic function over a  constraint set which is the sum of a compact set $K$ and a polyhedral cone $U$, both located in a finite-dimensional Euclidean space. A related result appears in \cite[Satz 2.3]{Krabs1972jour} (infimal value of a rather specific functional in a normed space setting). The notion of convergence in all of the above-mentioned cases is either convergence with respect to the Hausdorff distance or a slightly more general notion. 

Yet another result appears in \cite[Corollary 4]{LentCensor1991jour}. It essentially says that the infimum of  a certain functional, the variable of which is a compact subset of a finite-dimensional Euclidean space (up to slight change of notation, this is essentially the $\gamma$ functional mentioned in \cite[Relation (9)]{LentCensor1991jour}), is sequentially continuous with respect to the Kuratowski-Painlev\'e notion of convergence of sets. 

Finally, we also note that Corollary \bref{cor:A_t} below, regarding the continuity, and  Lipschitz continuity, of the so-called ``optimal value function'' or ``marginal function'', has predecessors, and we mention them in the beginning of Section  \bref{sec:ParamOpt}. 
\end{enumerate}
\end{remark}

\section{Application 1: continuity of the optimal value function from Parametric Optimization}\label{sec:ParamOpt}
In this section we show how Theorem \bref{thm:HausdorffSupInf} can be used to prove that 
the so-called ``optimal value function'' (or ``marginal function'', or ``inf-projection'') from parametric optimization is continuous under certain assumptions. Corollary \bref{cor:A_t} below extends, to the setting of pseudo-distance spaces, related results formulated in a metric space setting, such as \cite[Lemma 3.18, Parts 3 and 4]{LudererMinchenkoSatsura2002book}, \cite[Lemma 1]{KlatteKummer1985incol} and \cite[Theorem 3.1.22 and Proposition 3.3.10]{PallaschkeRolewicz1997book}. 

\begin{cor}\label{cor:A_t}
Let $(X,d_X)$ and $(I,d_I)$ be two arbitrary pseudo-distance spaces. Suppose that $\mathscr{C}$ is a nonempty set of nonempty subsets of $X$.  Assume that $(A_t)_{t\in I}$ is a family of subsets in $\mathscr{C}$. Given $f:X\to\R$, define $\phi^*:I\to[-\infty,\infty]$ and $\phi_*:I\to[-\infty,\infty]$ by $\phi^*(t):=\SUP_f(A_t):=\sup\{f(x): x\in A_t\}$ and $\phi_*(t):=\INF_f(A_t):=\inf\{f(x): x\in A_t\}$, respectively, for each $t\in I$. Then the following two statements hold:
\begin{enumerate}[(a)]
\item\label{item:f_is_uniformly_continuous} Assume that $\mathscr{C}$ has  the property that for each $A\in \mathscr{C}$, there is $r_A>0$ such that $B[A,r_A]\in \mathscr{C}$. Assume also that $f$ is uniformly continuous on each $A\in \mathscr{C}$, and for each $t_0\in I$, one has $\limsup_{t\to t_0}D_{H,X}(A_{t_0},A_t)\leq 0$ in the sense of Section \bref{sec:Preliminaries} (here $D_{H,X}$ is the Hausdorff distance induced by $d_X$ and not by $d_I$). If $A_t\in\dom(\SUP_f)$ for all $t\in I$, then $\phi^*$ is a continuous function from $I$ to $\R$, and if $A_t\in\dom(\INF_f)$ for every $t\in I$, then $\phi_*$ is a continuous function from $I$ to $\R$. 

\item\label{item:f_is_Lipschitz_continuous} Assume that $\mathscr{C}$ has  the following two properties: first, that for each $C_1, C_2\in \mathscr{C}$, there is $C_{1,2}\in\mathscr{C}$ such that $C_1\cup C_2\subseteq C_{1,2}$, and second, that $f$ is $\Lambda(C)$-Lipschitz continuous on each $C\in\mathscr{C}$ for some $\Lambda(C)>0$. Assume also  that for all $t,s\in I$ there is some $\alpha_{t,s}>0$ such that $D_{H,X}(A_t,A_s)\leq\alpha_{t,s} d_I(t,s)$. 

If $A_t\in\dom(\SUP_f)$ for all $t\in I$, then for all $t,s\in I$ there is some $\Lambda_{t,s}>0$ such that $|\phi^*(t)-\phi^*(s)|\leq \alpha_{t,s}\Lambda_{t,s}d_I(t,s)$. In particular, if $A_t\in\dom(\SUP_f)$ for all $t\in I$, and $f$ is $\Lambda$-Lipschitz continuous on $X$ and there is some $\alpha>0$ such that $D_{H,X}(A_t,A_s)\leq\alpha d_I(t,s)$ for all $t,s\in I$, then $\phi^*$ is $\alpha\Lambda$-Lipschitz continuous on $I$. Similarly, if $A_t\in\dom(\INF_f)$ for every $t\in I$, then for all $t,s\in I$ there is some $\Lambda_{t,s}>0$ such that $|\phi_*(t)-\phi_*(s)|\leq \alpha_{t,s}\Lambda_{t,s}d_I(t,s)$. In particular, if $A_t\in\dom(\INF_f)$ for every $t\in I$, and $f$ is $\Lambda$-Lipschitz continuous on $X$ and there is some $\alpha>0$ such that $D_{H,X}(A_t,A_s)\leq\alpha d_I(t,s)$ for all $t,s\in I$, then $\phi_*$ is $\alpha\Lambda$-Lipschitz continuous on $I$.
\end{enumerate}
\end{cor}

\begin{proof}
We start with Part \beqref{item:f_is_uniformly_continuous} regarding $\phi^*$. Let $t_0\in I$ and $\epsilon>0$ be arbitrary. Since we assume that $B[A_{t_0},r_0]\in \mathscr{C}$ for some $r_0>0$, it follows, in particular, that $B[A_{t_0},r_0]\neq\emptyset$. Moreover, if $\emptyset\neq A'\subseteq X$ satisfies $D_{H,X}(A_{t_0},A')<r_0$, then $A'\subseteq B[A_{t_0},r_0]$, as follows from the definition of $B[A_{t_0},r_0]$ and \beqref{eq:D_H}. Since $f$ is uniformly continuous on any subset which belongs to $\mathscr{C}$, it is uniformly continuous on $B[A_{t_0},r_0]$. Moreover, by our assumption, $A_{t_0}\in \dom(\SUP_f)$. Thus, we can apply Theorem \bref{thm:HausdorffSupInf}, where the space there is $B[A_{t_0},r_0]$ and the pseudo-distance is the restriction of $d_X$ to $B[A_{t_0},r_0]^2$. This theorem  implies, in particular, that for our $\epsilon$, there exists $\delta_0\in (0,r_0]$ such that for all $\emptyset\neq A'\subseteq B[A_{t_0},r_0]$ satisfying $A'\in\dom(\SUP_f)$ and $D_{H,X}(A_{t_0},A')<\delta_0$, we have 
\begin{equation}\label{eq:uA'}
|\phi^*(t_0)-\SUP_f(A')|<\epsilon. 
\end{equation}
Since we assume that $\limsup_{t\to t_0}D_{H,X}(A_{t_0},A_t)\leq 0$, it follows that for the positive number $\delta_0$, there exists $\delta>0$ such that for each $t\in I$ which satisfies $d_I(t_0,t)<\delta$, we have $D_{H,X}(A_{t_0},A_t)<\delta_0$. Since $A_t\in\dom(\SUP_f)$, we can substitute $A':=A_t$ in \beqref{eq:uA'}. In other words, $|\phi^*(t_0)-\phi^*(t)|<\epsilon$ for each $t\in I$ satisfying $d_I(t_0,t)<\delta$. Therefore $\phi^*$ is continuous at $t_0$. The proof of the claim regarding $\phi_*$ is similar. 

Now we prove Part \beqref{item:f_is_Lipschitz_continuous} regarding $\phi^*$. Let $t,s \in I$ be given. According to our assumptions, there is some $A_{t,s}\in\mathscr{C}$ and $\Lambda_{t,s}>0$ such that $A_t\cup A_s\subseteq A_{t,s}$ and $f$ is $\Lambda_{t,s}$-Lipschitz continuous on $A_{t,s}$. Since we also assume that there is some $\alpha_{t,s}>0$ such that $D_{H,X}(A_t,A_s)\leq\alpha_{t,s} d_I(t,s)$, Theorem \bref{thm:HausdorffSupInf} (in which the space is $A_{t,s}$ and the pseudo-distance is the restriction of $d_X$ to $A_{t,s}^2$) implies that 
\begin{equation}\label{eq:Lambda_alpha_t,s}
|\phi^*(t)-\phi^*(s)|=|\SUP_f(A_t)-\SUP_f(A_s)|\leq\Lambda_{t,s} D_{H,X}(A_t,A_s)\leq \alpha_{t,s} \Lambda_{t,s} d_I(t,s).
\end{equation}
In particular, if $f$ is $\Lambda$-Lipschitz continuous on $X$ for some $\Lambda>0$ and there is some $\alpha>0$ such that $D_{H,X}(A_t,A_s)\leq \alpha d_I(t,s)$ for all $t,s\in I$, then we conclude from Theorem \bref{thm:HausdorffSupInf} (as in \beqref{eq:Lambda_alpha_t,s}) that $\phi^*$ is $\alpha\Lambda$-Lipschitz continuous on $I$. The proof of the claim regarding $\phi_*$ is similar. 
\end{proof}

\begin{remark}\label{ex:mathscrC}
An example of a set $\mathscr{C}$ having the property mentioned in the formulation of Corollary \bref{cor:A_t}\beqref{item:f_is_uniformly_continuous}  is provided by the set of all nonempty and bounded subsets of a metric space: in Corollary \bref{cor:f'Lip} below we use this example. A second example for $\mathscr{C}$ is the set of all nonempty subsets of a metric space, or, more generally, the set of all nonempty subsets of a pseudo-distance space $(X,d)$, such that $(X,d)$ has the property that $0=d(x,x)\leq d(x,y)$ for every $(x,y)\in X^2$ (for instance, this happens if $X$ is a real or complex vector space and $d$ is the distance induced by a Minkowski functional of a convex subset of $X$ which contains the origin: in Section \bref{sec:LinearNonlinear} below we use this set;  another example: $X$ is the zone of a Bregman function and $d$ is the associated Bregman divergence \cite{Bregman1967jour, ReemReichDePierro2019jour(BregmanEntropy)}). A third example is as follows: again, we consider a pseudo-distance space $(X,d)$ having the property that $0=d(x,x)\leq d(x,y)$ for every $(x,y)\in X^2$, and are also given a uniformly continuous function $f:X\to\R$; then Theorem \bref{thm:HausdorffSupInf} and Remark \beqref{rem:FurtherCommentsStability}\bref{item:redundant} imply that we can take $\mathscr{C}$ to be either $\dom(\SUP_f)$ or $\dom(\INF_f)$.

Similarly, if the space $X$ is a metric space and $f$ is Lipschitz continuous on every nonempty and bounded subset of the space, then the set  of all nonempty and bounded subsets of the space is an example for a set $\mathscr{C}$ having the property mentioned in the formulation of Corollary \bref{cor:A_t}\beqref{item:f_is_Lipschitz_continuous}. If $f$ is Lipschitz continuous on the whole space, then the set of all nonempty subsets of the space is an example for $\mathscr{C}$. Two additional examples for $\mathscr{C}$ are the set of all nonempty  convex subsets of a normed space, assuming that $f$ is Lipschitz continuous on every nonempty convex subset of the space, and the set of all nonempty, convex and bounded subsets of the space, assuming that $f$ is Lipschitz continuous on every nonempty, convex and bounded subset of the space. 

As a final remark, we note that in  Corollary \bref{cor:A_t}\beqref{item:f_is_Lipschitz_continuous}, if $d_I(t,s)\in\R$ for all $t,s\in I$, then both  $\alpha_{t,s}$ and $\alpha$ can be assigned the value 0. 
\end{remark}

\section{Application 2:  mixed linear-nonlinear programming, extended generalized inverses, a general variant of Hoffman's Lemma}\label{sec:LinearNonlinear}
In this section we consider two mixed linear-nonlinear programming problems and establish continuity properties of the corresponding optimal value functions. In the first case (Example \bref{ex:LinearNonlinearWholeSpace}) the objective function is, in general, nonlinear but the constraints are linear, and in the second case (Example \bref{ex:LinearNonlinearPartSpace}) the function is nonlinear and the constraints are partly linear and partly nonlinear. Our results extend partly, but significantly, a  theory which was developed in previous works: see Remark \bref{rem:LinearNonlinear} below for more details. Along the way we present a very general variant of the so-called Hoffman's Lemma (see Lemma \bref{lem:LipschitzAffine}) and generalize the so-called ``generalized inverse'' of a linear operator (Definition \bref{def:GenGenInv}, Remarks \bref{rem:GenInVHoff}--\bref{rem:GenInvAsymmetricNorm}). We present our results in Subsection \bref{subsec:ApplicationLipschitzAffine}. Before presenting them, we  need some background, in the form of a few definitions and remarks, which are presented in Subsection \bref{subsec:Background}. 

\subsection{Background}\label{subsec:Background}
In this subsection we discuss some concepts which we use later. 
We start by recalling the concept of a Minkowski functional. 
\begin{defin}\label{def:Minkowski}
Suppose that $X$ is a vector space over $\R$ or $\C$. The Minkowski functional associated with a convex subset $C\subseteq X$ which contains 0  is the function $\M_C:X\to[0,\infty]$ defined by 
\begin{equation}\label{eq:gamma_C}
\M_C(x):=\inf\{\mu: \mu \geq 0\,\textnormal{and}\, x\in\mu C\},\quad x\in X,
\end{equation}
where, of course, $\mu C=\{\mu c: c\in  C\}$ and $\inf\emptyset:=\infty$. 
\end{defin}
It is well known that $\M_C$ might be a norm, but unless $C$ satisfies certain properties, $\M_C$ is not a norm in general (for instance, it may be asymmetric and may attain the value $+\infty$). Nevertheless, $\M_C$ enjoys several properties similar to those of a norm, for example it is positively homogeneous (that is, $\M_C(\lambda x)=\lambda \M_C(x)$ for all $x\in X$ and $\lambda>0$)  and subadditive (namely, $\M_C(x+y)\leq \M_C(x)+\M_C(y)$ for all $x,y\in X$). See \cite[p. 26]{VanTiel1984book} for more details (note that there additional assumptions are imposed on $C$, but the proofs hold when one merely assumes that $C$ is convex and $0\in C$, and the assertions hold also in the cases where $\M_C$ attains the value $+\infty$; in fact, one can impose even weaker assumptions for the assertions to hold).

Now we introduce an extension of the concept of a generalized inverse of a linear operator. 
\begin{defin}\label{def:GenGenInv}
Let $X$ and $Y$ be nonempty sets. Assume that on both sets a binary operation $+$ is defined. Here we abuse our notation and denote both operations by ``+'').  Assume further that the $+$ operation on $Y$ has a right-neutral element $0$, that is, $y+0=y$ for each $y\in Y$. Given $L:X\to Y$, suppose that its kernel $\{x\in X: Lx=0\}$ is nonempty. Assume  that $\wt{Y}$ is a subset of $Y$ which contains the image $L(X)$ of $X$ by $L$. The operator $L$ is said to have an ``extended generalized inverse'' (EGI) with domain $\wt{Y}$ if there exists a (not necessarily additive) operator $\wt{L}:\wt{Y}\to X$ which has the following property: for all $t\in L(X)$ there are $v\in X$ and $a_{0,v}$ in the kernel of $L$ such that $t=Lv$ and $\wt{L}(t)=v+a_{0,v}$. 
\end{defin}

\begin{remark}\label{rem:GenInVHoff}
Definition \bref{def:GenGenInv} significantly generalizes the concept of a ``generalized inverse'' of a linear operator (also called the ``generalized reciprocal'', or the ``Moore-Penrose generalized inverse'', or the ``Moore-Penrose pseudo-inverse'', or the ``Moore-Penrose inverse'', in honor of the contributions of Moore \cite{Moore1920jour} and Penrose \cite{Penrose1955jour} to this theory). Indeed, consider first the case of a standard finite-dimensional generalized inverse. Here one starts with a linear operator $L:X\to Y$ which acts between two real or complex finite-dimensional vector spaces $X$ and $Y$, and its generalized inverse is the  linear operator $L^{\dag}: Y\to X$,  which uniquely exists, and satisfies the following four relations: $LL^{\dag}L=L$, $L^{\dag}LL^{\dag}=L^{\dag}$, $(LL^{\dag})^*=LL^{\dag}$ and $(L^{\dag}L)^*=L^{\dag}L$, where $*$ denotes the linear conjugation. Since the generalized inverse $L^{\dag}$ automatically satisfies $L^{\dag}L=id_X-P_{A_0}=P_{L^*(Y)}$ (this follows from \cite[Ex. 58, p. 80]{Ben-IsraelGreville2003book} and \cite[Theorem 1, p. 12]{Ben-IsraelGreville2003book}), where $L(X)$ is the image/range of $L$, $A_0$ is kernel of $L$, $P_{A_0}$ is the projection operator onto $A_0$ along $L^*(Y)$, and $id_X$ is the identity operator on $X$, our extended generalized inverse $\wt{L}$ does generalize the concept of (standard) generalized inverse. 

More generally, suppose that $L$ is a bounded linear operator which acts between two Banach spaces $X$ and $Y$ and has the following two properties: first, its kernel $A_0$ satisfies $X=A_0\oplus A_0'$ for some closed linear subspace $A_0'$ of $X$, and second, its image $L(X)$  satisfies $Y=L(X)\oplus M'$ for some closed linear subspace $M'$ of $Y$ (namely,  $A_0$ is topologically complemented in $X$ and $L(X)$ is topologically complemented in $Y$). A (Moore-Penrose) generalized inverse to $L$ is a bounded linear  operator $L^{\dag}:Y\to X$ which satisfies the following four relations: $LL^{\dag}L=L$, $L^{\dag}LL^{\dag}=L^{\dag}$, $LL^{\dag}=P_{L(X)}$ and $L^{\dag}L=id_X-P_{A_0}$, where $P_{A_0}$ is the projection operator onto $A_0$ along $A_0'$ and $P_{L(X)}$ is the projection operator onto $L(X)$ along $M'$. Such an operator sometimes exists. It can be seen that in Definition  \bref{def:GenGenInv} we required only a relation weaker than the fourth relation, and we did not require $\wt{L}$ to be linear, $L$ to be bounded, and the spaces to be Banach spaces (or even vector spaces). 

 We note that the concept of generalized inverse of a linear operator has been extensively studied during the last 60 years or so and has found  various applications: see, for instance, the books \cite{Ben-IsraelGreville2003book, CampbellMeyer1991book, DjordjevicRakocevic2008book, Groetsch1977book, WangWeiQiao2018book} and the semi-survey \cite{Nashed1987jour} (and the references therein); see also the online bibliographic list \cite{Ben-Israel2001List} which is composed of 1670 items, and last updated in June 2001. Hence we believe that our extension of this concept has a promising potential to yield diverse applications too. We also note that other generalizations of this concept exist (see, for example, \cite{Ben-IsraelGreville2003book, XueCao2014jour} and the references therein), but  they seem to have a somewhat different nature than our generalization. 
\end{remark}

\begin{remark}\label{rem:ExistenceGenInv} 
Consider the setting of Definition \bref{def:GenGenInv}. Assume further that $X$ has a right-neutral element 0 (that is, $x+0=x$ for each $x\in X$). If $L(0)=0$ (as, in particular, happens if $L$ is additive), then $L$ has an EGI. Indeed, given $t\in L(x)$, there is at least one $v_t\in X$ such that $Lv_t=t$.  Now we use the Axiom of Choice to define, for each $t\in \wt{Y}:=L(X)$, an operator $\wt{L}:\wt{Y}\to X$ by  $\wt{L}t:=v_t$. Obviously $\wt{L}t=v_t+0$, and since $L(0)=0$, one can see that the required conditions mentioned in Definition \bref{def:GenGenInv} are satisfied.  
\end{remark}

\begin{remark}\label{rem:GenInvAsymmetricNorm}
For applications, one usually needs to require more from an extended generalized inverse $\wt{L}$ of $L$, such as its Lipschitz continuity.  Here we want to mention two cases in which one can provide a linear and  bounded EGI in spaces which are not  necessarily normed spaces. 
\begin{enumerate}[(i)] 
\item In this case we show that $\wt{L}$ can be taken to be a standard (Moore-Penrose) generalized inverse $L^{\dag}$, where the main technical work is  to show that $L^{\dag}$ is Lipschitz continuous with respect to the given pseudo-distances. Our basic setting is two Banach spaces $(X,\|\cdot\|_X)$ and $(Y,\|\cdot\|_Y)$ and a bounded linear operator $L:X\to Y$ which is known to have a standard  generalized inverse $L^{\dag}:Y\to X$ which is a bounded  linear operator. We assume further that $\M_{C_X}:X\to [0,\infty]$ is a Minkowski functional induced by a convex subset $C_X\subseteq X$ which contains 0 in its interior (interior with respect to $\|\cdot\|_X$),  and $\M_{C_Y}:Y\to [0,\infty]$ is a Minkowski functional induced by a convex subset $C_Y\subseteq Y$ which is bounded with respect to the norm $\|\cdot\|_Y$ and contains 0. It can be checked immediately that $\M_{C_X}(x)\leq \alpha\|x\|_X$ for all $x\in X$, where $\alpha:=1/r$ and $r>0$ is the radius of some open ball (with respect to $\|\cdot\|_X$) which is centered at 0 and is contained in $C_X$, and $\|y\|_Y\leq \beta\M_{C_Y}(y)$ for all $y\in Y$, where $\beta:=\sup\{\|c\|_Y: c\in C_Y\}$. Hence for all $y, z\in Y$, 
\begin{equation*}
\M_{C_X}(L^{\dag}y-L^{\dag}z)\leq \alpha\|L^{\dag}y-L^{\dag}z\|_X\leq \alpha\|L^{\dag}\|\|y-z\|_Y\leq\alpha\|L^{\dag}\|\beta\M_{C_Y}(y-z), 
\end{equation*}
and so $\wt{L}:=L^{\dag}$ is $\alpha\|L^{\dag}\|\beta$-Lipschitz continuous as a function from the pseudo-distance space $(X,\M_{C_X})$ to the pseudo-distance space $(Y,\M_{C_Y})$. 

\item\label{item:sigma}
As a second example, suppose that $(X,\|\cdot\|_X)$ is a real or complex normed space. Assume that $(Y,\|\cdot\|_Y)$ is a normed space (over $F=\R$ or $F=\C$).  Assume that $\scS_X:X\to [0,\infty]$ has the property that for some $\kappa>0$ and all scalars $\alpha$ and $\beta$ (real or complex), 
\begin{equation}\label{eq:S_Xinequality}
\scS_X(\alpha u+\beta v)\leq \kappa\left(|\alpha|\scS_X(u)+|\beta|\scS_X(v)\right),\quad \forall u, v\in X, 
\end{equation}
where we regard $0\cdot \infty$ as $0$. For instance, the previously  mentioned inequality holds if $\scS_X$ is a norm or a quasi-norm; it also holds if $\scS_X$  is the  Minkowski functional induced by a convex subset $C_X$ of $X$ which contains 0 and has the property that $C_X\subseteq \sigma (-C_X)$ for some $\sigma>0$ if the field associated with $X$ is real (if the field associated with $X$ is complex, then \beqref{eq:S_Xinequality} holds when $C_X\subseteq \sigma (iC_X)$ for some $\sigma>0$; this assumption implies, in particular, that $C_X\subseteq \sigma^2 (-C_X)$). 

Suppose that $L:X\to Y$ is a linear operator having the property that its  range $\wt{Y}:=L(X)$ is finite-dimensional. In what follows we assume that $\wt{Y}\neq\{0\}$, since otherwise the analysis becomes rather trivial. Since $\wt{Y}$ is finite-dimensional, it has a basis $(w_j)_{j=1}^{n}$, where $n\in \N$ and for all $j\in \{1,\ldots,n\}$ there is some $v_j\in X$ such that $w_j=L(v_j)$. Let $V:=\Span\{v_j: j\in\{1,\ldots,n\}\}$ and $L_V:V\to \wt{Y}$ be the restriction of $L$ to $V$. Then $L_V$ is invertible because $(w_j)_{j=1}^{n}$ is a basis of $\wt{Y}$, and $L_V^{-1}$ is bounded since it is a linear mapping which acts between the finite-dimensional spaces $\wt{Y}$ and $V$. Denote $\wt{L}:=L_V^{-1}$. 

We assume that $\scS_X$ has the property that $\scS_X(\wt{L}w_j)<\infty$ for all $j\in\{1,\ldots,n\}$ (this happens, if, for example, $\scS_X(x)<\infty$ for all $x\in X$, but it can happen in other cases as well). Assume further that $\scS_Y:Y\to[0,\infty]$ has the properties that the set $S_{\wt{Y}}:=\{y\in \wt{Y}: \scS_Y(y)=1\}$ is bounded with respect to $\|\cdot\|_Y$, that $\scS_Y$ is positively homogeneous on $\wt{Y}$, and $\scS_Y(y)>0$ for all $0\neq y\in \wt{Y}$. These properties hold if, for example, $\scS_Y$ is the Minkowski functional induced by a convex subset $E$ of $Y$ which contains 0 and is closed and bounded with respect to $\|\cdot\|_Y$ (in this case one also has $E=\{y\in Y: \scS_Y(y)\leq 1\}$, as follows, for instance, essentially word for word from the proof of \cite[Lemma 3.1(a)]{ReemReich2018jourJMAA}). Another example for  $\scS_Y$ are functions which induce certain star bodies; such functions appear in the theory of geometry of numbers \cite{Cassels1997, GruberLek, Mahler1946, Mordell1945}; for instance, if $Y=\R^3$, then  $\scS_Y$, defined for all $y=(y_1,y_2,y_3)\in Y$ by 
\begin{equation*}
\scS_Y(y):=(|y_1||y_2||y_3|)^{1/3}+(|y_2|^2|y_3|^3)^{1/5}+\sqrt{|y_3^2-y_1^2|}+(\sqrt{|y_1|}+\sqrt{|y_2|}+\sqrt{|y_3|})^2,
\end{equation*}
has the required properties on the whole space $Y$.

 We claim that under the above-mentioned assumptions, $\wt{L}$ is Lipschitz continuous as a linear operator from $(\wt{Y},\scS_Y)$ to $(X,\scS_X)$.  Indeed, any $y\in \wt{Y}$ can be written uniquely as $y=\sum_{j=1}^{n}\lambda_j w_j$, where each $\lambda_j$, $j\in\{1,\ldots,n\}$, is a scalar in $F$. Let $T:F^{n}\to \wt{Y}$ be the linear operator defined by $T(\mu_j)_{j=1}^{n}:=\sum_{j=1}^n\mu_j w_j$ for every $(\mu_j)_{j=1}^n\in F^n$. Then $T$ is invertible, and since $T^{-1}$ acts between two finite-dimensional spaces, $T^{-1}$ is bounded (where the norm on $F^n$ is the max norm $\|\cdot\|_{\infty}$). Since $S_{\wt{Y}}$ is nonempty (because $y/\scS_Y(y)\in S_{\wt{Y}}$ for all $0\neq y\in \wt{Y}$) and bounded with respect to $\|\cdot\|_Y$, it follows that $T^{-1}(S_{\wt{Y}})$ is bounded in $F^n$, namely $\eta:=\sup\{\|T^{-1}(s)\|_{\infty}: s\in S_{\wt{Y}}\}<\infty$. This inequality, \beqref{eq:S_Xinequality}, induction (on $n$) and the assumption that $\scS_X(\wt{L}w_j)<\infty$ for each $j\in\{1,\ldots,n\}$, imply that for all $y\in S_{\wt{Y}}$, 
\begin{multline}\label{eq:sigma}
\scS_X(\wt{L}y)=\scS_X(\wt{L}(\sum_{j=1}^n\lambda_j w_j))=\scS_X(\sum_{j=1}^n\lambda_j \wt{L}w_j)\\
\leq \tau\sum_{j=1}^n|\lambda_j|\scS_X(\wt{L}w_j)\leq\tau\eta\sum_{j=1}^n\scS_X(\wt{L}w_j)=:\sigma<\infty,
\end{multline}
where $\tau:=\max\{\kappa,\ldots,\kappa^{n-1}\}$. Since $S_{\wt{Y}}=\{y\in \wt{Y}: \scS_{\wt{Y}}(y)=1\}$, our assumptions on $\scS_Y$ imply that for all $y,z\in \wt{Y}$ which satisfy $y\neq z$, we have $\scS_Y(y-z)>0$ and $(y-z)/\scS_Y(y-z)\in S_{\wt{Y}}$. Hence, using  \beqref{eq:S_Xinequality}, it follows from \beqref{eq:sigma} that 
\begin{equation*}
\scS_X(\wt{L}y-\wt{L}z)\leq \kappa\scS_Y(y-z)\scS_X\left(\wt{L}\left(\frac{y-z}{\scS_Y(y-z)}\right)\right)\leq \kappa\sigma \scS_Y(y-z).
\end{equation*}
 In other words, $\wt{L}$ is $\kappa\sigma$-Lipschitz continuous, as claimed. We note that when $\scS_X$ and $\scS_{Y}$ are norms, then we can, of course, take also $\|\wt{L}\|$ as a Lipschitz constant of $\wt{L}$. Finally, since, given $t\in \wt{Y}$, there is $v\in V$ such that $Lv=t$, and since $\wt{L}t=L_V^{-1}L_Vv=v+0$, it follows that $\wt{L}$ is an EGI.
\end{enumerate}
\end{remark}

\subsection{The application}\label{subsec:ApplicationLipschitzAffine}
In this subsection we present the application of our stability results to parametric optimization. We do this in Example \bref{ex:LinearNonlinearWholeSpace} and Example \bref{ex:LinearNonlinearPartSpace}, which are also based on Lemma \bref{lem:LipschitzAffine} below. This lemma is related to, but different from, the so called ``Hoffman's Lemma'' and several variants of it: see Remark \bref{rem:Hoffman} below for more details. 

\begin{lem}\label{lem:LipschitzAffine} {\bf (A general variant of Hoffman's Lemma)}
Let $X$ be a commutative additive group and let $\scS_X:X\to [-\infty,\infty]$ be given. Let $d:X^2\to[-\infty,\infty]$ be the pseudo-distance induced by $\scS_X$, namely $d(u,v):=\scS_X(u-v)$ for all $u,v\in X$. Let $Y$ be an additive group and assume that $L:X\to Y$ is an additive operator which has an extended generalized inverse $\wt{L}:\wt{Y}\to X$ with domain $\wt{Y}$, where $\wt{Y}$ is an additive subgroup  of $Y$ which satisfies $I:=L(X)\subseteq \wt{Y}$. Given $\scS_{\wt{Y}}:\wt{Y}\to [-\infty,\infty]$, suppose that $\wt{L}$ is $\alpha$-Lipschitz continuous from $(\wt{Y},\scS_{\wt{Y}})$ to $(X,\scS_{X})$ for some $0\neq \alpha\in\R$ (the case $\alpha=0$ is allowed too if $\scS_{\wt{Y}}(y)\in\R$ for all $y\in\wt{Y}$). Denote $A_t:=\{x\in X: Lx=t\}$ for all $t\in I$. Then $D_H(A_s,A_t)\leq \max\{\alpha\scS_{\wt{Y}}(s-t),\alpha\scS_{\wt{Y}}(t-s)\}$ for all $s, t\in I$. 
\end{lem}

\begin{proof}
Let $t\in I$ be given. Then, according to the assumption that $\wt{L}$ is an EGI of $L$ (Definition \bref{def:GenGenInv}),  there are $v\in X$ and $a_{0,v}\in A_0$ such that $t=Lv$ and $\wt{L}t=v+a_{0,v}$. Hence $v=\wt{L}t-a_{0,v}$. The definition of $A_t$ implies that $v\in A_t$, and therefore $\wt{L}t-a_{0,v}$ is in $A_t$. But, as is well known and can be proved immediately, $A_t$, which  is the set of solutions $z\in X$ to the inhomogeneous additive equation $Lz=t$, is the sum of $A_0$ and an arbitrary particular solution to the inhomogeneous equation. Since, as shown above, $\wt{L}t-a_{0,v}$ solves the inhomogeneous equation $Lz=t$, we have  $A_t=A_0+(\wt{L}t-a_{0,v})$, and so $A_t=A_0+\wt{L}t$ since $X$ is commutative and $A_0$ is invariant under translation. Similarly, $A_s=A_0+\wt{L}s$ for all $s\in I$. Fix some $s\in I$ and let $x\in A_s$. Then $x=a_{0,x}+\wt{L}s$ for some $a_{0,x}\in A_0$. Let $y:=a_{0,x}+\wt{L}t$. Then $y\in A_0+\wt{L}t$, namely, $y\in A_t$. In addition, since $\wt{L}$ is $\alpha$-Lipschitz  continuous on $\wt{Y}$ with respect to $\scS_{\wt{Y}}$ and $\scS_X$, we have $\scS_{X}(\wt{L}(\wt{x})-\wt{L}(\wt{y}))\leq \alpha\scS_{\wt{Y}}(\wt{x}-\wt{y})$ for all $\wt{x}$ and $\wt{y}$ in $\wt{Y}$. Since $I\subseteq \wt{Y}$, we conclude from the previous lines and the fact that $X$ is commutative that $\scS_X(x-y)=\scS_X(\wt{L}(s)-\wt{L}(t))\leq \alpha\scS_{\wt{Y}}(s-t)$. Since $y\in A_t$, we have $d(x,A_t)\leq \scS_X(x-y)$, and so $d(x,A_t)\leq \alpha\scS_{\wt{Y}}(s-t)$. But $x$ was an arbitrary element in $A_s$, and so $D_{asyH}(A_s,A_t)=\sup_{x\in A_s}d(x,A_t)\leq \alpha\scS_{\wt{Y}}(s-t)$. Similarly, $D_{asyH}(A_t,A_s)\leq \alpha\scS_{\wt{Y}}(t-s)$. Thus $D_{H}(A_s,A_t)\leq \max\{\alpha\scS_{\wt{Y}}(s-t),\alpha\scS_{\wt{Y}}(t-s)\}$. 
\end{proof}

\begin{remark}\label{rem:MoreGeneralLipschitzAffine} {\bf (An even more general variant of Hoffman's Lemma)}
It is possible to generalize Lemma \bref{lem:LipschitzAffine} even further. For instance, suppose that $\wt{L}$ has the property that for some function $\nu:I^2\to[-\infty,\infty]$, we have $\scS_X(\wt{L}s-\wt{L}t)\leq \nu(s,t)$ for all $(s,t)\in I^2$. Then an analysis similar to the one used in the  proof of Lemma \bref{lem:LipschitzAffine}  shows that $D_{H}(A_s,A_t)\leq \max\{\nu(s,t),\nu(t,s)\}$ for all $s,t \in I$. 
\end{remark}

\begin{remark}\label{rem:Hoffman} 
Lemma \bref{lem:LipschitzAffine} (and Remark \bref{rem:MoreGeneralLipschitzAffine}) is related to, but different from, some known results, such as \cite[Theorem 4.15, Theorem 4.16]{Daniel1973jour},  \cite[Theorem 1]{WalkupWets1969b-jour} and \cite[Theorem 2.1]{ZhengNg2004}. All of these results are variants of the so-called Hoffman's Lemma \cite[the Theorem in Section 2]{Hoffman1952jour}. 

The differences between Lemma \bref{lem:LipschitzAffine} and the above-mentioned results are both in the formulations and in the methods of proof. 
For instance, the nature of the above-mentioned results is finite-dimensional (either the spaces are finite-dimensional or there are finite systems of equalities/inequalities) and no EGI appears (in fact, in the above-mentioned works, only in \cite[Theorem 4.15, Theorem 4.16]{Daniel1973jour} one can see the appearance of the standard generalized inverse of a matrix, but its use, in \cite[Lemma 4.1]{Daniel1973jour}, is significantly different from our use of the EGI in Lemma \bref{lem:LipschitzAffine}); in addition, in these works the spaces are always assumed to be normed spaces, although in Hoffman's paper \cite{Hoffman1952jour} one allows a more general magnitude function (which, implicitly, has to satisfy certain  relations with respect to the Euclidean norm); on the other hand, the nature of Lemma \bref{lem:LipschitzAffine} is rather general (possibly infinite-dimensional vector spaces or even spaces which are not vector spaces, magnitude functions which are much more general than a norm, and so on) and we use (and introduce) the concept of an extended generalized inverse.  

We note that \cite[Theorem 7]{Robinson1972jour} is also a general variant of Hoffman's Lemma, but its nature is different from Lemma \bref{lem:LipschitzAffine}. For instance, it is for convex processes in Banach spaces, that, when restricted to the linear case, some restrictive assumptions are imposed on the operator, unless the spaces are finite-dimensional. In addition, the method of proof in  \cite[Theorem 7]{Robinson1972jour} is significantly different from the one of Lemma \bref{lem:LipschitzAffine}, and, in particular, it does not use the notion of an  EGI (however, interestingly, it does use a certain inverse, namely an inverse of a set-valued operator). 

Anyway, since Lemma \bref{lem:LipschitzAffine} is related to the above-mentioned results, one can regard it as a very general variant of Hoffman's Lemma.
\end{remark}

\begin{exmp}\label{ex:LinearNonlinearWholeSpace}{\bf (Nonlinear function, linear constraints)}
Consider the setting of Lemma \bref{lem:LipschitzAffine}.  Assume that $f:X\to\R$ is bounded from below and is also uniformly continuous with respect to the distance induced by the ``pseudo-magnitude'' $\scS_X:X\to [0,\infty]$ (not necessarily a Minkowski functional).  Denote by $d_I$ the distance on $I$  induced by $\scS_{\wt{Y}}$. Suppose that $\scS_{\wt{Y}}(0)=0$. Suppose also that the conjugate pseudo-magnitude $\overline{\scS_{\wt{Y}}}:X\to[0,\infty]$, which is defined by $\overline{\scS_{\wt{Y}}}(x):=\scS_{\wt{Y}}(-x)$ for each $x\in X$, is continuous at 0 (the continuity is with respect to $\scS_{\wt{Y}}$).  Let $\phi_*:I\to\R$ be the optimal value function defined by $\phi_*(t):=\INF_f(A_t)$, $t\in I$, namely 
\begin{equation*}
\phi_*(t)=\inf\{f(x): x\in X,\, Lx=t\}, \quad t\in L(X).
\end{equation*} 
Let $\mathscr{C}$ be the set of all nonempty subsets of $X$. Since $f$ is bounded from below and $A_t\neq \emptyset$ for every $t\in I$, it follows that $A_t\in\dom(\INF_f)$ for each $t\in I$. Moreover, as we have shown in Lemma \bref{lem:LipschitzAffine}, there exists some real number $\alpha$, which must be nonnegative since $\scS_{\wt{Y}}$ is nonnegative (unless $\scS_{\wt{Y}}$ vanishes identically, but in this case we can obviously replace $\alpha$ by a nonnegative number), such that $D_H(A_s,A_t)\leq  \max\{\alpha\scS_{\wt{Y}}(s-t),\alpha\overline{\scS_{\wt{Y}}}(s-t)\}$ for all $s, t\in I$. Since $\overline{\scS_{\wt{Y}}}$ is continuous at 0 and $\overline{\scS_{\wt{Y}}}(0)=0=\scS_{\wt{Y}}(0)$, we have  $\lim_{t\to s}D_H(A_s,A_t)=0$ for all $s\in I$. Since the assumptions on $\scS_{\wt{Y}}$ and hence on $d_I$ imply that $\mathscr{C}$ satisfies the conditions needed in Corollary \bref{cor:A_t}\beqref{item:f_is_uniformly_continuous},  this corollary implies that $\phi_*$ is continuous. If, in addition, $f$ is $\Lambda$-Lipschitz continuous on $X$ for some $\Lambda>0$ and $\overline{\scS_{\wt{Y}}}$ is $\beta$-Lipschitz continuous with respect to $\scS_{\wt{Y}}$  for some $\beta>0$, then $\phi_*$ is $\alpha\cdot\max\{1, \beta\}\cdot \Lambda$-Lipschitz continuous, as follows from Corollary \bref{cor:A_t}\beqref{item:f_is_Lipschitz_continuous}.
\end{exmp}

\begin{remark}\label{rem:DiscontinuousAsymmetricNorm}
In connection with Example \bref{ex:LinearNonlinearWholeSpace}, it could be of interest to note that functions which are uniformly continuous with respect to a pseudo-magnitude (even with respect to a Minkowski functional) may not be uniformly continuous with respect to a norm $\|\cdot\|$. Indeed, let $X:=\R^2$ with the usual Euclidean norm $\|\cdot\|$,  and let $C:=[-2,1]\times\{0\}$. Consider the Minkowski functional $\M_C$ induced by $C$. A simple calculation shows that  $\M_C(x_1,x_2)=x_1$ if $x=(x_1,x_2)\in [0,\infty)\times\{0\}$, $\M_C(x_1,x_2)=-0.5x_1$ if $(x_1,x_2)\in (-\infty,0]\times\{0\}$ and $\M_C(x_1,x_2)=\infty$ otherwise. In particular, $\M_C(0)=0$ and $0\leq \overline{\M_C}(x)\leq 2\M_C(x)$ at any $x\in X$, and hence the conjugate $\overline{\M_C}$ (see Example \bref{ex:LinearNonlinearWholeSpace}) is continuous at 0 with respect to $\M_C$. 

Now let $f:X\to\R$ be defined for all $(x_1,x_2)\in \R^2$ by $f(x_1,x_2):=g(x_2)$, where $g:\R\to\R$ is some discontinuous function (with respect to the usual absolute value metric). It follows that $f$ is not uniformly continuous with respect to the Euclidean norm, since it is not even continuous with respect to this  norm. To see that $f$ is uniformly continuous with respect to $\M_C$, let $\epsilon>0$ be  given and let $\delta:=\epsilon$. Given any $(x_1,x_2)$ and $(y_1,y_2)$ in $X$ which satisfy $\M_C((x_1,x_2)-(y_1,y_2))<\delta$, one has, in particular,  $\M_C((x_1-y_1,x_2-y_2))<\infty$. Hence $x_2=y_2$. Since $f(x_1,x_2)=g(x_2)$ and $f(y_1,y_2)=g(y_2)$, it follows that $|f(x_1,x_2)-f(y_1,y_2)|=0<\epsilon$, and so indeed $f$ is uniformly continuous. 
\end{remark}

\begin{exmp}\label{ex:LinearNonlinearPartSpace} {\bf (Nonlinear function, mixed linear-nonlinear constraints)}
Given $n\in\N$, endow $\R^n$ with some norm $\|\cdot\|$ and suppose that $C$ is compact and convex, and has a  nonempty interior $\Int(C)$.  Assume that $f:C\to\R$ is continuous. Given some $m\in\N$, suppose that $L:\R^n\to \R^m$ is a linear operator, where the norm on $\R^m$ is $\|\cdot\|_{\R^m}$ (just an arbitrary norm). For all $t$ in the image $L(\R^n)$ of $\R^n$ by $L$, denote $E_t:=\{x\in \R^n: Lx=t\}$. Denote by $I$ the set of all $t\in L(\R^n)$ for which the following regularity condition holds: the affine subspace $E_t$ intersects $\Int(C)$, namely there is some $x\in\Int(C)$ such that $Lx=t$. Assume that $I\neq\emptyset$ (as shown in Remark \bref{rem:LinearNonlinear}\beqref{item:IisOpen} below, this latter assumption actually implies that $I$ is convex and open in $L(\R^n)$). For all $t\in I$, let $A_t:=E_t\cap C$ and let $\phi_*:I\to\R$ be the optimal value function defined by $\phi_*(t):=\INF_f(A_t)$, $t\in I$, that is,
\begin{equation*}
\phi_*(t)=\inf\{f(x): x\in C,\, Lx=t\}, \quad t\in I.
\end{equation*} 
Given $t\in I$, since $E_t\cap\Int(C)\neq\emptyset$, we have $A_t\neq\emptyset$, and since $E_t$ is closed and $C$ is compact, $A_t$ is compact. Thus the Extreme Value Theorem implies that $A_t\in\dom(\INF_f)$ for each $t\in I$. We prove below that $\lim_{t\to s}D_H(A_s,A_t)=0$ for all $s\in I$. This fact, when combined with the fact that the continuous function $f$ is actually uniformly continuous (since $C$ is compact), implies that we may use  Corollary \bref{cor:A_t}\beqref{item:f_is_uniformly_continuous} (in which the first space is $C$ with the restriction of $\|\cdot\|$ to $C$ as the distance function, $d_I$ is the distance induced on $I$ by $\|\cdot\|_{\R^m}$, and $\mathscr{C}$ is the set of all nonempty subsets of $C$), from which we conclude that $\phi_*$ is continuous. 

To see that indeed $\lim_{t\to s}D_H(A_s,A_t)=0$ for all $s\in I$, suppose to the contrary that this is not true. Then there are $s\in I$, $\epsilon>0$ and a sequence $(t_k)_{k=1}^{\infty}$ of elements in $I$  such that $\lim_{k\to\infty}t_k=s$ and $D_H(A_s,A_{t_k})\geq \epsilon$ for each $k\in\N$. Therefore either $D_{asyH}(E_{t_k}\cap C,E_s\cap C)\geq \epsilon$ for all $k\in N_1$, where $N_1\subseteq \N$ is an infinite set, or $D_{asyH}(E_s\cap C,E_{t_k}\cap C)\geq\epsilon$ for all $k\in N_2$, where $N_2\subseteq \N$ is an infinite set. 

Consider the first case. It implies that for each $k\in N_1$, there is some $x_k\in E_{t_k}\cap C$ such that $d(x_k,E_s\cap C)>0.5\epsilon$. Since the sequence $(x_k)_{k\in N_1}$ is contained in the compact set $C$, there is an infinite set $N_{11}\subseteq N_1$ and $x_{\infty}\in C$ such that $\lim_{k\in N_{11}}x_k=x_{\infty}$. Lemma \bref{lem:LipschitzAffine}, Remark \bref{rem:GenInvAsymmetricNorm}\beqref{item:sigma} and the assumption that $\lim_{k\to\infty}t_k=s$ imply that $\lim_{k\to\infty}D_H(E_{t_k},E_s)=0$, and so from \beqref{eq:D_asyH} we conclude that $\lim_{k\to\infty, k\in N_{11}}d(x_k,E_s)=0$. This fact and the continuity of the distance function imply that $d(x_{\infty},E_s)=0$. Thus (since $E_s$ is closed) $x_{\infty}\in E_s$. We conclude that $x_{\infty}\in E_s\cap C$. However, since $d(x_k,E_s\cap C)>0.5\epsilon$ for each $k\in N_1$, when we pass to the limit $k\to\infty$, $k\in N_{11}$ and use the continuity of the distance function, we have $d(x_{\infty},E_s\cap C)\geq 0.5\epsilon$, a contradiction. 

Now consider the second case which was mentioned two paragraphs earlier. It implies that for each $k\in N_2$, there is some $y_k\in E_{s}\cap C$ such that $d(y_k,E_{t_k}\cap C)>0.5\epsilon$. Since the sequence $(y_k)_{k\in N_2}$ is contained in the compact set $E_s\cap C$, there is an infinite set $N_{22}\subseteq N_2$ and $y_{\infty}\in E_s\cap C$ such that $\lim_{k\in N_{22}}y_k=y_{\infty}$. According to our assumption, $E_s\cap\Int(C)\neq\emptyset$. Let $z\in E_s\cap \Int(C)$ and $r:=\min\{0.1\epsilon,0.5\|y_{\infty}-z\|\}$. If $y_{\infty}=z$, then $r=0$. We let $w:=z$ and observe that $w\in \Int(C)$. Otherwise, let $w:=y_{\infty}+r((z-y_{\infty})/\|z-y_{\infty}\|)$. In this latter case $w$ belongs to the half-open line segment $(y_{\infty},z]$.  As is well known \cite[Theorem 6.1, p. 45]{Rockafellar1970book}, since $z\in\Int(C)$ and $y_{\infty}\in C$, the convexity of $C$ implies that  $(y_{\infty},z]$ is contained in $\Int(C)$. Thus $w\in \Int(C)$ again. Hence there is some $\rho\in (0,r)$ such that the ball of radius $\rho$ with center $w$ is contained in $C$. Moreover, $w\in E_s$ since $E_s$ is convex and $[y_{\infty},z]\subseteq E_s$. 

Lemma \bref{lem:LipschitzAffine} and Remark \bref{rem:GenInvAsymmetricNorm}\beqref{item:sigma} imply that $\lim_{k\to\infty}D_H(E_{t_k},E_s)=0$. As a result, it follows that $\lim_{k\to\infty}d(w,E_{t_k})=0$. Thus for all $k\in\N$ sufficiently large there is some $u_k\in E_{t_k}$ such that $\|w-u_k\|<\rho$. Our choice of $\rho$ implies that $u_k\in C$ and the triangle inequality implies that $\|y_{\infty}-u_k\|\leq \|y_{\infty}-w\|+\|w-u_k\|<r+\rho<2r\leq 0.2\epsilon$. Therefore 
\begin{equation}\label{eq:y_inftyEtkC}
d(y_{\infty},E_{t_k}\cap C)\leq \|y_{\infty}-u_k\|<0.25\epsilon
\end{equation}
for all $k\in\N$ sufficiently large and, in particular, for all $k\in N_{22}$ sufficiently large. On the other hand,  the fact that $\lim_{k\in N_{22}}y_k=y_{\infty}$ implies that $\|y_{\infty}-y_k\|<0.25\epsilon$ for all $k\in N_{22}$ sufficiently large. Since $d(y_k,E_{t_k}\cap C)>0.5\epsilon$ for each $k\in N_2$, the  triangle inequality implies that $0.5\epsilon<d(y_k,E_{t_k}\cap C)\leq \|y_k-y_{\infty}\|+d(y_{\infty},E_{t_k}\cap C)<0.25\epsilon+d(y_{\infty},E_{t_k}\cap C)$. Thus $0.25\epsilon<d(y_{\infty},E_{t_k}\cap C)$ for all $k\in N_{22}$ sufficiently large. This inequality contradicts \beqref{eq:y_inftyEtkC} and proves that the second case mentioned several paragraphs above cannot hold too. Thus we indeed have $\lim_{t\to s}D_H(A_s,A_t)=0$ for all $s\in I$, as asserted.
\end{exmp}

Below we collect a few remarks regarding Examples \bref{ex:LinearNonlinearWholeSpace} and \bref{ex:LinearNonlinearPartSpace}.
\begin{remark}\label{rem:LinearNonlinear}
\begin{enumerate}[(i)]
\item Examples \bref{ex:LinearNonlinearWholeSpace} and \bref{ex:LinearNonlinearPartSpace} extend partly, but significantly,  the stability theory developed in \cite[pp. 279--280]{Wets2003jour} and \cite[Theorem 2]{WalkupWets1969jour} (for a  related theory, see \cite[p. 281]{Wets2003jour} and \cite[Lemma 4.1]{RomischWets2007jour}). This theory has been applied to analyzing stochastic programs \cite[Section 4]{WalkupWets1969jour}, \cite[Section 4]{RomischWets2007jour}. The setting in \cite[pp. 279--280]{Wets2003jour} and \cite[Theorem 2]{WalkupWets1969jour} is a finite-dimensional Euclidean space, a polyhedral set $C$, and an objective function $f$ which is Lipschitz continuous on a set which contains $C$; another requirement in \cite[pp. 279--280]{Wets2003jour} is that either the level-sets of $f$ are bounded or $C^{\infty}\cap A_0=\{0\}$, where $A_0$ is the kernel of the linear operator $L$ and $C^{\infty}$ is the horizontal cone associated with $C$. It is proved in \cite{Wets2003jour} that the associated optimal value function $\phi_*$  is Lipschitz continuous under these assumptions (note: the constraint set there is written as $\{x\in C: Lx=b-t\}$, where $b$ is a given vector and $t$ is the parameter; thus by a simple change of variable we can arrive at this formulation). It can be seen that  Examples \bref{ex:LinearNonlinearWholeSpace} and \bref{ex:LinearNonlinearPartSpace} extend this theory to the case of spaces which are not necessarily normed spaces (and not necessarily finite-dimensional), a constraint set $C$ which is either the entire space or a (usually non-polyhedral) convex body, a linear operator $L$ which should have a Lipschitz continuous EGI, and an objective function $f$ which is either Lipschitz continuous or merely uniformly continuous. We are still able to derive the continuity of $\phi_*$ under these conditions, and sometimes (Example \bref{ex:LinearNonlinearWholeSpace}) its Lipschitz continuity. 

 We believe that the theory developed in this section can be extended further, and, in particular, that it is possible to remove (at least in some interesting cases) the compactness and finite-dimensionality assumptions from Example \bref{ex:LinearNonlinearPartSpace}. 

\item\label{item:IisOpen} The set $I$ mentioned in Example \bref{ex:LinearNonlinearPartSpace} is actually open and convex whenever it is nonempty. Indeed, given $s\in I$, let $x\in E_s\cap \Int(C)$. In particular, $x\in \Int(C)$ and hence there exists some $r>0$ such that the ball of radius $r$ about $x$ is contained in $\Int(C)$. Remark \bref{rem:GenInvAsymmetricNorm}\beqref{item:sigma} and  Lemma \bref{lem:LipschitzAffine} imply that there is some $\alpha>0$ such that $D_H(E_s,E_t)\leq \alpha \|s-t\|_{\R^m}$ for all $t\in L(\R^n)$. Thus, if $\delta:=r/\alpha$, then  for each $t\in L(\R^n)$ which satisfies $\|s-t\|_{\R^m}<\delta$, we have $D_H(E_s,E_t)<r$. This inequality and the fact that $x$ also belongs to $E_s$ imply that there is  some $y\in E_t$ such that $\|x-y\|<r$. Thus $y\in \Int(C)\cap E_t$ and so $t\in I$ for all $t$ in the ball of radius $\delta$ about $s$. To see that $I$ is convex, let  $s,t\in I$ and $\lambda\in [0,1]$ be given. Then $Lx_s=s$ and $Lx_t=t$ for some $x_s, x_t\in \Int(C)$. We have $\lambda x_s+(1-\lambda)x_t\in \Int(C)$ since $\Int(C)$ is convex. In addition, $L(\lambda x_s+(1-\lambda)x_t)=\lambda L(x_s)+(1-\lambda)L(x_t)=\lambda s+(1-\lambda)t$. Consequently, $\lambda s+(1-\lambda)t\in I$, as required. 
\end{enumerate}
\end{remark}

\section{Application 3: a sequence of Lipschitz constants}\label{sec:ContinuityLipCont}
In this section we use Corollary \bref{cor:A_t} in order to show that, under some assumptions, a rather general sequence of positive numbers can be a sequence of Lipschitz constants associated with a given function (each Lipschitz constant corresponds to a certain subset on which one measures the Lipschitz continuity of the function). Corollary \bref{cor:f'Lip} below has recently been applied  in the analysis of a telescopic proximal gradient method \cite{ReemReichDePierro2019teprog}. 

\begin{cor}\label{cor:f'Lip}
Suppose that $f:U\to \R$ is a twice continuously (Fr\'echet) differentiable function defined on an open and convex subset $U$  of some real normed space $(X,\|\cdot\|)$, $X\neq\{0\}$. Suppose that $C$ is a convex subset of $X$ which has the property that $C\cap U\neq\emptyset$. Assume that $f''$ is bounded and uniformly continuous on  bounded subsets of $C\cap U$. Fix an arbitrary $y_0\in C\cap U$, and let $s:=\sup\{\|f''(x)\|: x\in C\cap U\}$ and $s_0:=\|f''(y_0)\|$. If $s=\infty$, then for each strictly increasing sequence $(\lambda_k)_{k=1}^{\infty}$ of positive numbers which satisfies $\lambda_1>s_0$ and $\lim_{k\to\infty}\lambda_k=\infty$, there exists an increasing sequence $(S_k)_{k=1}^{\infty}$ of bounded and convex subsets of $C$ (and also closed if $C$ is closed) which satisfies the following properties: first, $S_k\cap U\neq\emptyset$ for all $k\in \N$, second, $\cup_{k=1}^{\infty}S_k=C$, third, for each $k\in\N$, the function $f'$ is  Lipschitz continuous on $S_k\cap U$ with $\lambda_k$ as a Lipschitz constant; moreover, if $C$ contains more than one point, then also $S_k\cap U$ contains more than one point for each $k\in\N$. Finally, if $s<\infty$, then $f'$ is Lipschitz continuous on $C\cap U$ with $s$ as a Lipschitz constant. 
\end{cor}

\begin{proof}
Suppose first that $s=\infty$. Let $I:=[0,\infty)$ and let $d_I$ be the standard absolute value metric. For each $t\in I$, define $B_t$ to be the intersection of $C$ with the closed  ball of radius $t$ and center $y_0$ (here $B_0:=\{y_0\}$). Then $B_t$ is a bounded and convex subset of $C$ for each $t\in I$, and it is also closed if $C$ is closed. Let $A_t:=B_t\cap U$ for each $t\in I$. Then $A_t\neq\emptyset$ (it contains $y_0$), convex and bounded for all $t\in I$. In addition, $\cup_{t\in I}B_t=C$ and $\cup_{t\in I}A_t=C\cap U$. An immediate verification shows that $D_H(A_t,A_{t'})\leq |t-t'|$ for all $t,t'\in I$. Since $f''$ exists and is bounded on bounded subsets of $C\cap U$, the function $h$,  which is defined by  $h(x):=\|f''(x)\|$, $x\in C\cap U$, is finite at each point, and it is also bounded on each bounded subset of $C\cap U$. This implies  that the function $\phi^*:I\to (-\infty,\infty]$, which is defined by $\phi^*(t):=\SUP_h(A_t)$ for each $t\in I$, satisfies $\phi^*(t)\in [0,\infty)$ for each $t\in I$. In addition, since $f''$ is uniformly continuous on bounded subsets of $C\cap U$, the triangle inequality shows that the function $h$, too, is uniformly continuous on bounded subsets of $C\cap U$. We conclude from the  previous lines that the conditions needed in  Corollary \bref{cor:A_t}\beqref{item:f_is_uniformly_continuous} hold (the first pseudo-distance space there is $C\cap U$, where the pseudo-distance is the metric which is induced by the restriction of the norm of $X$ to $C\cap U$; in addition, the set $\mathscr{C}$ in  Corollary \bref{cor:A_t}\beqref{item:f_is_uniformly_continuous} is the set of all nonempty and bounded subsets of $C\cap U$), and consequently, $\phi^*$ is a continuous function on $I$. 

Since $s=\sup\{\|f''(x)\|: x\in C\cap U\}=\infty$, for each $\rho\geq \|f''(y_0)\|=s_0$, there exists $x\in C\cap U$ such that $\|f''(x)\|>\rho$. Since $\cup_{t\in I}A_t=C\cap U$, there exists $t(x)\in I$ such that $x\in A_{t(x)}$. As a result, from the definition of $\phi^*$ we see that $\phi^*(t(x))\geq \|f''(x)\|>\rho$. By applying the classical Intermediate Value Theorem to the continuous function $\phi^*$ on the interval $[0,t(x)]$, we conclude that each value between $\phi^*(0)=s_0$ and $\phi^*(t(x))$ is attained. In particular, $\rho$ is attained. Since $\rho$ was an arbitrary number which is greater than or equal to $s_0$ and since $\phi^*$ is increasing, it follows that the image of $I=[0,\infty)$ under $\phi^*$ is the interval $[s_0,\infty)$. Therefore, given $k\in\N$, since  $\lambda_k\geq \lambda_1>s_0$, there exists $t_k\in [0,\infty)$ such that  $\phi^*(t_k)=\lambda_k$, and  this $t_k$ must be positive, otherwise $t_k=0$ and hence $s_0=\phi^*(0)=\phi^*(t_k)=\lambda_k$, a contradiction. 

Let $S_k:=B_{t_k}$ for each $k\in \N$. Then $S_k$ is bounded and convex for each $k\in\N$, and it is also closed if $C$ is closed. In addition, $S_k\cap U$ is nonempty (it contains $y_0$), bounded and convex for every $k\in\N$. Since $\|f''(x)\|\leq \sup\{\|f''(y)\|: y\in S_k\cap U\}=\phi^*(t_k)$ for all $x\in S_k\cap U$, and since $f'$ is continuously differentiable on $U$ and hence on $S_k\cap U$, the (generalized) Mean Value Theorem applied to $f'$ (see \cite[Theorem 1.8, p. 13, and also p. 23]{AmbrosettiProdi1993book}; this theorem is formulated for G\^ateaux differentiable functions acting between real Banach spaces, but it holds as well for Fr\'echet  differentiable functions acting between real normed spaces, because no completeness assumption is needed in the proof, and the Fr\'echet and G\^ateaux derivatives coincide in our case) implies that $f'$ is Lipschitz continuous on $S_k\cap U$ with $\phi^*(t_k)$ as a Lipschitz constant, namely with $\lambda_k$ as a  Lipschitz constant. 

Now we show that $\cup_{k=1}^{\infty}S_k=C$. Indeed, since $\phi^*$ is increasing and $(\lambda_k)_{k=1}^{\infty}$ is strictly increasing, it follows that $(t_k)_{k=1}^{\infty}$ is increasing. Hence $\ell:=\lim_{k\to\infty}t_k$ exists and it must be that $\ell=\infty$, otherwise $\lambda_k=\phi^*(t_k)\leq \phi^*(\ell)<\infty$ for all $k\in\N$, a contradiction to the assumption that $\lim_{k\to\infty}\lambda_k=\infty$. Hence the union of the closed balls with common center $y_0$ and radii $t_k$, $k\in\N$, is $X$. Thus the intersection of this union with $C$ is $C$ itself. On the other hand, this intersection is $\cup_{k=1}^{\infty}S_k$, as follows from the definition of the subsets $S_k$, $k\in\N$. In other words, $\cup_{k=1}^{\infty}S_k=C$.

It remains to show that if $C$ contains more than one point, then  $S_k\cap U$ also contains more than one point for every $k\in\N$. Indeed, take some arbitrary $w_0\in C$ which satisfies $w_0\neq y_0$. The line segment $[y_0,w_0]$ is contained in $C$ because $C$ is convex. Since $U$ is open and $y_0\in U$, there is a sufficiently  small closed ball $B$ of center $y_0$ and positive radius $r<\min\{\|y_0-w_0\|,t_k\}$ such that $B\subset U$. The intersection of $B$ with $C$ contains the segment $[y_0,y_0+r\theta]$, where $\theta:=(w_0-y_0)/\|w_0-y_0\|$. Since  $r<t_k$, it follows from the definition of $S_k$ that $B\cap C\subset B_{t_k}=S_k$. Hence $S_k$ contains the nondegenerate segment $[y_0,y_0+r\theta]$, namely it contains more than one point. 

Finally, we need to consider the case where $s<\infty$. In this case $\|f''(x)\|\leq s<\infty$ for every $x\in C\cap U$. Since $C\cap U$ is convex and $f'$ is Fr\'echet (hence G\^ateaux) differentiable on $U$, the Mean  Value Theorem applied to $f'$ implies that $f'$ is Lipschitz continuous on $C\cap U$ with $s$ as a Lipschitz constant. 
\end{proof}

\section{Application 4: a general scheme for tackling a wide class of nonconvex and nonsmooth optimization problems}\label{sec:NonconvexNonsmooth}

\subsection{The method:} Given a pseudo-distance space $(X,d)$, consider the general optimization problem of minimizing (or maximizing) a given uniformly continuous function $f:X\to \R$ over a nonempty subset $A\subseteq X$. Theorem \bref{thm:HausdorffSupInf} suggests a general scheme for approximating both $\INF_f(A)$ and $\SUP_f(A)$. Indeed, consider the case of approximating $\INF_f(A)$ (the case of approximating $\SUP_f(A)$ follows a similar reasoning) and assume that it is known that $\INF_f(A)\in \R$. Assume also that we are able to approximate $A$ by a sequence $(A_k)_{k=1}^{\infty}$ of subsets of $X$ such that $\lim_{k\to\infty}D_H(A,A_k)=0$ and $\INF_f(A_k)\in\R$ for all $k\in\N$. Furthermore, assume that we are also able to compute an approximation $\tilde{\sigma}_k$ to $\INF_f(A_k)$ so that $\lim_{k\to\infty}|\tilde{\sigma}_k-\INF_f(A_k)|=0$. Then Theorem \bref{thm:HausdorffSupInf} ensures that  $\lim_{k\to\infty}\tilde{\sigma}_k=\lim_{k\to\infty}[\tilde{\sigma}_k-\INF_f(A_k)]+\lim_{k\to\infty}\INF_f(A_k)=\INF_f(A)$. Consequently, the general scheme is nothing but computing $\tilde{\sigma}_1, \tilde{\sigma}_2, \tilde{\sigma}_3, \ldots$. 

\subsection{A few remarks:}
The above-mentioned method seems to be useful in cases where $A$ itself does not have a ``finite representation'' or is not easily computable. For example, suppose that $A$ is a component (or the union of the components) of a double zone diagram induced by finitely many sites contained in a convex body in a finite-dimensional strictly convex normed space, or, more generally, in a compact geodesic metric space $X$ which has the geodesic inclusion property \cite[Definition 3.1]{Reem2018jour}. A double zone diagram is an exotic geometric object which is defined to be a fixed point of a certain operator which acts on tuples of sets. While its existence is known in general \cite[Theorem 5.5]{ReemReich2009jour} and one can even represent explicitly one of the  double zone diagrams, this representation is not finite in the sense that it based on an infinite increasing union of known ``inner tuples'' of sets: see \cite[Theorem 5.2]{Reem2018jour} (this representation was observed before in \cite[Lemma 5.1]{AsanoMatousekTokuyama2007jour} in a simpler setting). Thus, if we want to estimate the distance from a given point $q\in X$ to $A$, namely to estimate $\INF_f(A)$ for $f(x):=d(q,x)$, $x\in X$ (this is a uniformly continuous function), then we can use the above-mentioned method since it is known \cite[Corollary 5.3]{Reem2018jour} that if $A_k$ denotes a component of the ``inner tuple'' in iteration number $k$ corresponding to $A$, $k\in \N$ (or the union of the components if $A$ itself is the union of components of the double zone diagram), then $\lim_{k\to\infty}D_H(A,A_k)=0$. 

As a second example, consider the problem of minimizing a continuous  function $f$ over a finite-dimensional Euclidean space, where the constraint set $A$ is induced by a finite system of convex inequalities. Assume further that it is known that $A$ is contained in some known (and possibly large) closed ball. In this case $f$ is automatically uniformly continuous on the ball. Estimating $A$ is not always a simple task, but in \cite{ButnariuCensor1991jour} one can find a method which does exactly this. More precisely, it produces, in finitely many steps, an inner and outer polytopial  approximations $Q'_{\epsilon}\subseteq A\subseteq Q''_{\epsilon}$ to $A$ having Hausdorff distance from each other (and hence also from $A$) which is not greater than a known tolerance parameter $\epsilon$. In particular, by letting $\epsilon_k:=1/k$ and $A_k:=Q'_{\epsilon_k}$ for each $k\in\N$, we have $\lim_{k\to\infty}D_H(A_k,A)=0$, and hence we can use the above-mentioned method in order to estimate $\INF_f(A)$.

As a final remark in this section, we note that the idea of estimating the optimal value of some function over a given constraint set $A$ by estimating it over an approximating set $A_k$ and taking the limit $k\to\infty$ appears in other works, such as \cite[p. 367]{Kummer1977jour} (in a very intuitive and brief form), in  \cite{LentCensor1991jour} (the setting there is a finite-dimensional Euclidean space and the approximation is with respect to inner and outer limits of compact sets), and in \cite[Subsections 1.1 and 1.2]{Reem(FunAnalysisArXiv)2019arxiv} (in the setting of an interval/box, or, more generally, in the setting of a compact metric space). 

\section*{Acknowledgments}
Part of the work of the first author was done when he was at the Institute of Mathematical and Computer Sciences (ICMC), University of S\~ao Paulo,  S\~ao Carlos, Brazil (2014--2016), and was supported by FAPESP 2013/19504-9.  The second author was partially supported by the Israel Science Foundation (Grants 389/12 and 820/17), by the Fund for the Promotion of Research at the Technion and by the Technion General Research Fund.  The third author thanks CNPq  grant 306030/2014-4 and FAPESP 2013/19504-9. 
All the authors would like to express their thanks to Yair Censor for  helpful discussions related to \cite{LentCensor1991jour} and to all the people who have provided us with anonymous feedback.  

\bibliographystyle{acm}
\bibliography{biblio}

\end{document}